\documentclass[twoside,reqno]{amsart}
\usepackage[utf8]{inputenc}
\usepackage{color}
\usepackage{refstyle}
\usepackage{url}
\usepackage{amsmath}
\usepackage{amsthm}
\usepackage{amssymb}
\usepackage{graphicx}
\usepackage[unicode=true,pdfusetitle,
 bookmarks=true,bookmarksnumbered=true,bookmarksopen=false,
 breaklinks=false,pdfborder={0 0 1},backref=false,colorlinks=true]
 {hyperref}

\makeatletter

\AtBeginDocument{\providecommand\Propref[1]{\ref{Prop:#1}}}
\AtBeginDocument{\providecommand\Secref[1]{\ref{Sec:#1}}}
\AtBeginDocument{\providecommand\Appref[1]{\ref{App:#1}}}
\AtBeginDocument{\providecommand\Subsecref[1]{\ref{Subsec:#1}}}
\AtBeginDocument{\providecommand\Lemref[1]{\ref{Lem:#1}}}
\AtBeginDocument{\providecommand\Corref[1]{\ref{Cor:#1}}}
\RS@ifundefined{subsecref}
  {\newref{subsec}{name = \RSsectxt}}
  {}
\RS@ifundefined{thmref}
  {\def\RSthmtxt{theorem~}\newref{thm}{name = \RSthmtxt}}
  {}
\RS@ifundefined{lemref}
  {\def\RSlemtxt{lemma~}\newref{lem}{name = \RSlemtxt}}
  {}

\numberwithin{equation}{section}
\numberwithin{figure}{section}
\theoremstyle{plain}
\newtheorem{thm}{\protect\theoremname}
\theoremstyle{remark}
\newtheorem{rem}[thm]{\protect\remarkname}
\theoremstyle{plain}
\newtheorem{prop}[thm]{\protect\propositionname}
\theoremstyle{definition}
\newtheorem{defn}[thm]{\protect\definitionname}
\theoremstyle{plain}
\newtheorem{lem}[thm]{\protect\lemmaname}
\theoremstyle{plain}
\newtheorem{cor}[thm]{\protect\corollaryname}

\@ifundefined{date}{}{\date{}}
\usepackage{textgreek}
\usepackage{graphicx}
\usepackage{amsfonts}
\usepackage{amsthm}
\allowdisplaybreaks
\usepackage{cases}
\usepackage{tikz-cd}
\usepackage{cancel}
\usepackage{faktor}

\usepackage{refstyle}
\newref{cor}{name = corollary~, names = corollaries~, Name = Corollary~, Names = Corollaries~}
\newref{conj}{name= conjecture~, names = conjectures~, Name = Conjecture~, Names = Conjectures~}
\newref{exa}{name= example~, names = examples~, Name = Example~, Names = Examples~}
\newref{rem}{name= remark~, names = remarks~, Name = Remark~, Names = Remarks~} 
\newref{thm}{name= theorem~, names = theorems~, Name = Theorem~, Names = Theorems~}
\newref{lem}{name= lemma~, names = lemmas~, Name = Lemma~, Names = Lemmas~}
\newref{def}{name= definition~, names = definitions~, Name = Definition~, Names = Definitions~}
\newref{fact}{name= fact~, names = facts~, Name = Fact~, Names = Facts~}
\newref{blackbox}{name= blackbox~, names = blackboxes~, Name = Blackbox~, Names = Blackboxes~}
\newref{sec}{name= section~, names = sections~, Name = Section~, Names = Sections~}
\newref{app}{name= appendix~, names = appendicies~, Name = Appendix~, Names = Appendicies~}
\newref{subsec}{name= subsection~, names = subsections~, Name = Subsection~, Names = Subsections~}
\newref{prop}{name= proposition~, names = propositions~, Name = Proposition~, Names = Propositions~}

\def\Xint#1{\mathchoice
{\XXint\displaystyle\textstyle{#1}}%
{\XXint\textstyle\scriptstyle{#1}}%
{\XXint\scriptstyle\scriptscriptstyle{#1}}%
{\XXint\scriptscriptstyle\scriptscriptstyle{#1}}%
\!\int}
\def\XXint#1#2#3{{\setbox0=\hbox{$#1{#2#3}{\int}$ }
\vcenter{\hbox{$#2#3$ }}\kern-.6\wd0}}

\def\dashint{\Xint-}

\usepackage{accents}

\usepackage{tensind}
\tensordelimiter{?}

\hypersetup{citecolor=purple,linkcolor=blue}

\newcommand\restr[2]{{% we make the whole thing an ordinary symbol
  \left.\kern-\nulldelimiterspace % automatically resize the bar with \right
  #1 % the function
  \vphantom{\big|} % pretend it's a little taller at normal size
  \right|_{#2} % this is the delimiter
  }}

\usepackage{extarrows}

\usepackage{enumerate}\usepackage{color}%\usepackage[margin=1in]{geometry}

\let\div\relax
\DeclareMathOperator\div{div}
\DeclareMathOperator\vol{vol}
\DeclareMathOperator\Div{div}
\DeclareMathOperator\curl{curl}

\DeclareMathOperator\supp{supp}

\newcommand{\dd}[0]{\partial}
\newcommand{\grad}[0]{\nabla}

\newcommand{\T}[0]{\mathbb T^d}

\theoremstyle{plain}

\newcommand{\Ig}[2]{I^{g;#1}_{#2}}

\makeatother

\providecommand{\corollaryname}{Corollary}
\providecommand{\definitionname}{Definition}
\providecommand{\lemmaname}{Lemma}
\providecommand{\propositionname}{Proposition}
\providecommand{\remarkname}{Remark}
\providecommand{\theoremname}{Theorem}

\begin{document}
\global\long\def\divop{\operatorname{div}}%
\global\long\def\supp{\operatorname{supp}}%
\global\long\def\curl{\operatorname{curl}}%

\global\long\def\grad{\nabla}%

\global\long\def\dd{\partial}%

\global\long\def\Ig{I_{\#2}^{g;\#1}}%

\global\long\def\T{\mathbb{T}^{N}}%

\setlength{\parskip}{1ex}

\title[Epochs of regularity for wild H\"older solutions]{Epochs of 
regularity for wild Hölder-continuous solutions of the Hypodissipative Navier-Stokes System}
\author{Aynur Bulut}
\address{Louisiana State University, 303 Lockett Hall, Baton Rouge, LA 70803}
\email{aynurbulut@lsu.edu}
\author{Manh Khang Huynh}
\address{Institute for Advanced Study, 1 Einstein Dr, Princeton, NJ 08540}
\email{hmkhang24@ias.edu}
\author{Stan Palasek}
\address{UCLA Department of Mathematics}
\email{palasek@math.ucla.edu}
\begin{abstract}
We consider the hypodissipative Navier-Stokes equations on $[0,T]\times\mathbb{T}^{d}$
and seek to construct non-unique, H\"older-continuous solutions with
epochs of regularity (smooth almost everywhere outside a small singular
set in time), using convex integration techniques. In particular,
we give quantitative relationships between the power of the fractional
Laplacian, the dimension of the singular set, and the regularity of
the solution. In addition, we also generalize the usual vector calculus
arguments to higher dimensions with Lagrangian coordinates.
\end{abstract}

\maketitle

\section{Introduction}

Fix $d\geq 3$.  We consider the hypodissipative Navier-Stokes equations
\begin{equation}
\begin{cases}
\partial_{t}v+(-\Delta)^{\gamma}v+\divop\left(v\otimes v\right)+\nabla p=0\\
\divop v=0
\end{cases}\label{eq:hypo_NS}
\end{equation}
on the periodic domain $\mathbb{T}^{d}$, where $0<\gamma<1$ denotes the strength of the
fractional dissipation, $v:[0,T]\times\mathbb{T}^{d}\to\mathbb{R}^{d}$
is the velocity field and $p:[0,T]\times\mathbb{T}^{d}\to\mathbb{R}$
is the pressure.

Recently, in the study of hydrodynamic turbulence, significant
attention has been directed towards problems such as Onsager's conjecture,
which roughly states that the kinetic energy of an ideal fluid may
fail to be conserved when the regularity is less than $\frac{1}{3}$.

The starting point for much of this work in recent years is a nonuniqueness
result, using ideas from convex integration, due to De Lellis and Sz\'ekelyhidi
Jr \cite{delellisEulerEquationsDifferential2007}.  A sequence of results, e.g. in
\cite{contiHPrincipleRigidityAlpha2009, DaneriS, BDLIS, IsettOh, isettProofOnsagerConjecture2018, buckmasterOnsagerConjectureAdmissible2017, buckmasterNonuniquenessWeakSolutions2018,buckmasterConvexIntegrationPhenomenologies2019}, and the references cited in these works, developed 
these ideas to tackle Onsager's conjecture.  In \cite{isettProofOnsagerConjecture2018}, Isett 
reached the conjectured threshold of $\frac{1}{3}-$ for the three-dimensional Euler equation 
on the torus, using Mikado flows and a delicate gluing technique.  Further developments include
Buckmaster--De Lellis--Sz\'ekelyhidi, Jr.--Vicol \cite{buckmasterOnsagerConjectureAdmissible2017}, which forms the main basis
for this work; we will refer to the strategy in \cite{buckmasterOnsagerConjectureAdmissible2017} as the {\it Onsager scheme}. The
scheme produces a weak solution that can attain any arbitrary energy
profile (this is sometimes referred to as {\it energy profile control}).

After this recent progress, the main techniques of convex integration have also been
used to construct various kinds of ``wild'' solutions (nonunique,
or failing to conserve energy) for the Euler equations, the Navier-Stokes
equations, as well as the fractional Navier-Stokes equations \cite{buckmasterNonuniquenessWeakSolutions2018,colomboIllposednessLeraySolutions2018,derosaInfinitelyManyLeray2019,cheskidovSharpNonuniquenessNavierStokes2020}.
For the Navier-Stokes equations, the dissipation term $(-\Delta)v$
can dominate the nonlinear term $\divop\left(v\otimes v\right)$,
and this presents a difficult obstruction to convex integration. At present,
this issue can be avoided by either using spatial intermittency (at
the cost of non-uniform control on the solution) or considering the
fractional Laplacian $(-\Delta)^{\gamma}$ instead. For an explanation
of intermittency, as well as more history and references, we refer
the interested readers to \cite{buckmasterConvexIntegrationPhenomenologies2019}.

One direction of research has looked into the construction
of wild solutions with {\it epochs of regularity} (that is, solutions that 
are smooth almost everywhere outside a temporal set of small dimension);
this was carried out for the hyperdissipative 
Navier-Stokes equations (using intermittency) in 
\cite{buckmasterWildSolutionsNavierStokes2020}, the Navier-Stokes
equations (using intermittency) in \cite{cheskidovSharpNonuniquenessNavierStokes2020},
and then for the Euler equations (not using intermittency) in \cite{derosaDimensionSingularSet2021}. 

We note that this goal stands in contradiction to the desire to have
energy profile control, since whenever the solution is smooth the energy
cannot increase.  These approaches make use of the Onsager scheme, with 
several refinements to the gluing approach of Isett \cite{isettProofOnsagerConjecture2018}, combined with estimates
on the overlapping (glued) regions.  Because energy correction is no longer 
required, the scheme is also simplified. 

In this paper, we look at the case of the hypodissipative Navier-Stokes
equations without using spatial intermittency, and try to determine
for which values of $\gamma$ in $(-\Delta)^{\gamma}$ one can construct
spatially Hölder-continuous solutions with epochs of regularity. In
addition, we also extend the arguments involving the Biot-Savart operator
and vector calculus (cf. the treatment in \cite{buckmasterOnsagerConjectureAdmissible2017})
to higher dimensions.

We now state our main theorem.
\begin{thm}
\label{thm:thm1}Fix $d\geq3$. Let $V_{1}$ and $V_{2}$ be smooth
solutions to (\ref{eq:hypo_NS}) such that $\int_{\mathbb{T}^{d}}\left(V_{1}-V_{2}\right)(t)=0$ for all $t$. 

For every positive $\beta,\gamma$ such that $\beta<\frac{1}{3}$
and $\beta+2\gamma<1$, there exist 
\[
\begin{cases}
v\in C_{t}^{0}C_{x}^{\beta-}\cap L_{t}^{1}C_{x}^{\beta_{1}}\\
\mathcal{B}\subset[0,T]\text{ closed}
\end{cases}
\]
where
\begin{itemize}
\item $v$ is a nonunique weak solution to (\ref{eq:hypo_NS}) given initial
data $V_{1}\left(0\right)$.
\item $v$ agrees with $V_{1}$ near $t=0$, and agrees with $V_{2}$ near
$t=T$
\item $\beta_{1}=\left(\frac{1-\beta}{2}\right)^{-}$, $\dim_{\mathrm{Hausdorff}}(\mathcal{B})\leq\left(\frac{1+\beta}{2(1-\beta)}\right)^{+}$ 
\item $v\big|_{\mathcal{B}^{c}\times\mathbb{T}^{d}}$ is smooth.
\end{itemize}
\end{thm}

In particular, Theorem \ref{thm:thm1} implies that, with what we currently know about the Onsager
scheme, the best fractional Laplacian we can handle (using only temporal
intermittency) is $(-\Delta)^{\frac{1}{2}-}$, which is quite a distance
away from the full Navier-Stokes equation. This confirms the heuristic that
without spatial intermittency, we want the dissipation term $(-\Delta)^{\gamma}v$
to be dominated by the nonlinear term $\divop\left(v\otimes v\right)$.
In addition, because $L_{t}^{\infty}C_{x}^{\beta}$ is supercritical
for the $\gamma$-hypodissipative Navier-Stokes equations when $\beta+2\gamma<1$,
we expect that this constraint is sharp.

The proof of Theorem \ref{thm:thm1} makes use of the strategy of the Onsager scheme, and in
particular follows from an iterative proposition based on the local existence theory, combined
with a modification of Isett's gluing technique to preserve the ``good'' temporal regions. 
The main difficulty is to optimize the length of the overlapping regions (where the cutoff 
functions meet). The iterative proposition is presented in Section \ref{sec:prelim}, where it is shown to imply Theorem \ref{thm:thm1}. The proof of the iterative proposition itself is deferred to Section \ref{sec:Proof-of-iteration}, where, after
a brief mollification argument, we reduce the issue to a series of technical estimates (first, a 
collection of estimates for the gluing construction, which we then treat in Section 
\ref{sec:Proof-of-gluing}; and then a perturbation result arising from convex integration, 
which we treat in Section \ref{sec:Convex-integration-and}).

\begin{rem}
As usual (see, e.g. \cite{derosaInfinitelyManyLeray2019,buckmasterOnsagerConjectureAdmissible2017}),
any $C_{t}^{0}C_{x}^{\alpha}$ solution with $\alpha\in\left(0,\frac{1}{3}\right)$
is automatically a $C_{t,x}^{\alpha-}$ solution. For any given $\beta\in\left(0,\frac{1}{3}\right)$,
we can construct a wild $v\in C_{t,x}^{\beta}$. 

For $\gamma<\frac{1}{3}$ and $\varepsilon\in\left(0,\frac{1}{3}\right)$,
by interpolation, this leads to the construction of wild solutions
in $C_{t}^{0}C_{x}^{\left(\frac{1}{3}-\varepsilon\right)^{-}}\cap L_{t}^{1}C_{x}^{\left(\frac{1}{3}+\frac{\varepsilon}{2}\right)^{-}}\cap L_{t}^{\frac{3}{2}-}C_{x}^{\frac{1}{3}}$,
with the singular set having dimension less than $1$, and to construction
of wild solutions in $C_{t}C_{x}^{0+}\cap L_{t}^{1}C_{x}^{\frac{1}{2}-}$,
with the dimension of the singular set bounded by $\frac{1}{2}+$. 

On the other hand, in the range $\frac{1}{3}\leq\gamma<\frac{1}{2},$
for each $\beta<1-2\gamma$, the dimension of the singular set is
bounded by 
\[
\left(\frac{1+\beta}{2(1-\beta)}\right)^{+}<\frac{1-\gamma}{2\gamma}.
\]

\end{rem}

\begin{figure}[h]
\centering{}\includegraphics[scale=0.4]{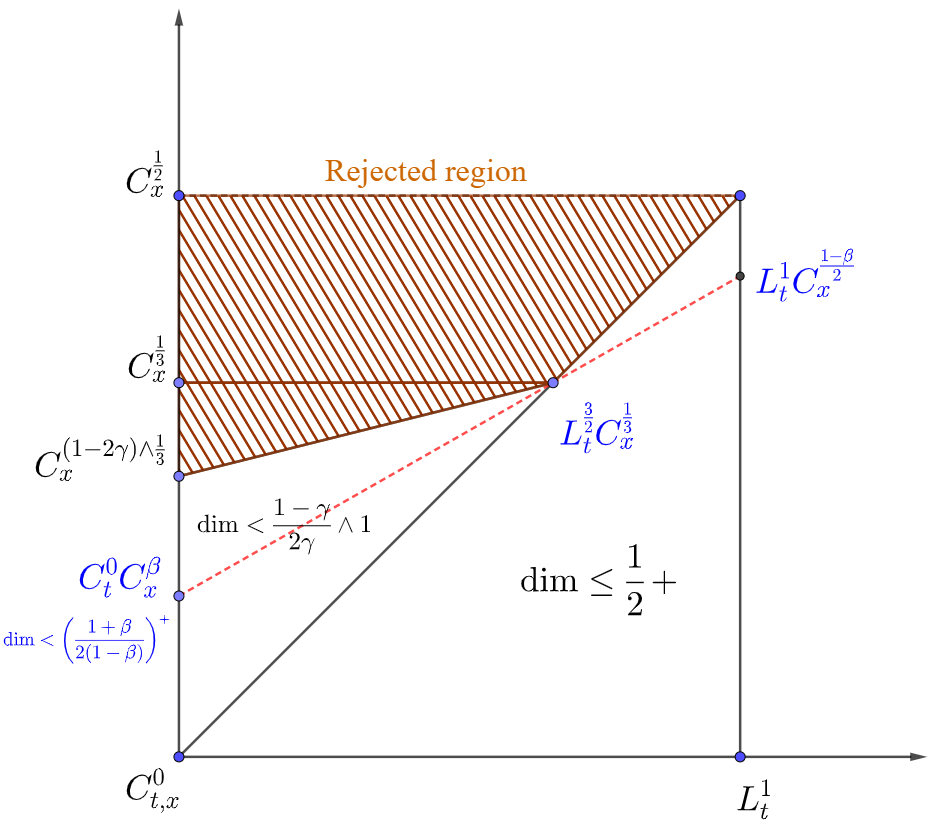}\caption{Regularity and dimension parameters. Given any $p\in\left(1,\infty\right)$
and $\mathcal{\widetilde{\beta}}$, we identify $L_{t}^{p}C_{x}^{\mathcal{\widetilde{\beta}}}$
as the point $\left(\frac{1}{p},\widetilde{\beta}\right)$. Assuming
it is away from the rejected region, by simple geometry, we can construct
a wild solution in $L_{t}^{p-}C_{x}^{\mathcal{\widetilde{\beta}}}$
or $L_{t}^{p}C_{x}^{\mathcal{\widetilde{\beta}}-}$ by constructing
one in $C_{t}^{0}C_{x}^{\beta-}\cap L_{t}^{1}C_{x}^{\left(\frac{1-\beta}{2}\right)-}$
where $\widetilde{\beta}\protect\leq\frac{1}{2p}+\left(1-\frac{3}{2p}\right)\beta$,
and get the corresponding $\mathrm{dim}_{\mathrm{Hausdorff}}\left(\mathcal{B}\right)$.
If $p=\frac{3}{2}$, we can arbitrarily choose $\beta<\left(1-2\gamma\right)\wedge\frac{1}{3}$. }
\end{figure}

\subsection*{Further comments and open questions}

The arguments we use to prove Theorem \ref{thm:thm1} immediately lead to an analogous
result for the Euler equations, since we treated $(-\Delta)^{\gamma}v$ as an 
error term.  In particular, in the proof of Theorem \ref{thm:thm1}, we show 
nonuniqueness for $C_{t}^{0}C_{x}^{\frac{1}{3}-}\cap L_{t}^{\frac{3}{2}-}C_{x}^{\frac{1}{3}}$
solutions. In the Euler context, this can be compared to the nonuniqueness of $L_{t}^{\frac{3}{2}-}C_{x}^{\frac{1}{3}}$
solutions in \cite[Theorem 1.10]{cheskidovSharpNonuniquenessNavierStokes2020}.  In
\cite{cheskidovSharpNonuniquenessNavierStokes2020}, rather than using the
Onsager scheme, the authors use spatial intermittency. As a consequence, the
solution they construct is not spacetime continuous; their singular set
$\mathcal{B}$ can have arbitrarily small Hausdorff dimension, and their scheme
also works in two dimensions.

Two open questions remain.  The first is to ask if can we further minimize the dimension of
the singular set $\mathcal{B}$, as suggested in \cite{derosaDimensionSingularSet2021}.
The second question of interest is to determine whether the construction can be adapted to 
construct solutions that obey some form of energy inequality.
Both questions lead to natural problems that we hope to consider in future works.

\subsection*{Outline of the Paper} In Section \ref{sec:prelim}, we specify our notational conventions
and introduce the main iterative scheme underlying the proof of Theorem \ref{thm:thm1}.  The iterative
step is formulated in Proposition \ref{prop:iterative}, which is then used to prove Theorem 
\ref{thm:thm1}.  The proof of Proposition \ref{prop:iterative} is the subject of Section \ref{sec:Proof-of-iteration}.  The proof
is reduced to two technical lemmas (a collection of gluing estimates, and a perturbation argument) which are treated in
Section \ref{sec:Proof-of-gluing} and Section \ref{sec:Convex-integration-and}, respectively.  A short appendix recalls several geometric preliminaries
used throughout the paper.

\subsection*{Acknowledgements}

The second author acknowledges that this material is based upon work
supported by a grant from the Institute for Advanced Study. The third
author acknowledges partial support from NSF grant DMS-1764034. The
authors also thank Camillo De Lellis and Alexey Cheskidov
for valuable discussions.

\section{\label{sec:prelim}Preliminaries and the iteration scheme}

We begin by establishing some notational conventions.  We will write $A\lesssim_{x,\neg y}B$ for $A\leq CB$, where $C$
is a positive constant depending on $x$ and not $y$. Similarly,
$A\sim_{x,\neg y}B$ means $A\lesssim_{x,\neg y}B$ and $B\lesssim_{x,\neg y}A$.
We will omit the explicit dependence when it is either not essential
or obvious by context.

For any real number $x$, we write $x+$ or $x^{+}$ to denote some
$y\in\left(x,x+\varepsilon\right)$ where $\varepsilon$ is some arbitrarily
small constant. Similarly we write $x-$ or $x^{-}$ for some $y\in\left(x-\varepsilon,x\right)$.

For any $N\in\mathbb{N}_{0}$ and $\alpha\in\left(0,1\right)$, we
write 
\begin{align*}
\left\Vert f\right\Vert _{N} & =\left\Vert f\right\Vert _{C^{N}},\quad\left[f\right]_{N}=\left\Vert \nabla^{N}f\right\Vert _{0},\quad\left[f\right]_{N+\alpha}=\left[\nabla^{N}f\right]_{C^{0,\alpha}},
\end{align*}
and
\begin{align*}
\left\Vert f\right\Vert _{N+\alpha} & =\left\Vert f\right\Vert _{C^{N,\alpha}}:=\left\Vert f\right\Vert _{N}+\left[f\right]_{N+\alpha},
\end{align*}
where $\left[\,\cdot\,\right]_{C^{0,\alpha}}$ denotes the H\"older seminorm.  We will often make use of the following elementary
inequality,
\[
\left\Vert fg\right\Vert _{r}\lesssim\left\Vert f\right\Vert _{0}\left[g\right]_{r}+\left[f\right]_{r}\left\Vert g\right\Vert _{0},
\]
which holds for any $r>0$.

\begin{defn}
For any $T>0$, $\nu>0$, vector field $v$ and $(2,0)$-tensor $R$ on
$[0,T]\times\mathbb{T}^{d}$, we say $(v,R)$ solves the $(\nu,\gamma,T)$-fNSR
equations (fractional Navier-Stokes-Reynolds) if there is a smooth pressure $p$ such that 
\begin{equation}
\begin{cases}
\dd_{t}v+\nu(-\Delta)^{\gamma}v+\divop v\otimes v+\grad p=\divop R\\
\divop v=0,
\end{cases}\label{eq:hypo_NSR}
\end{equation}
When $R=0$, we also say $v$ solves the $(\nu,\gamma,T)$-fNS equations.
\end{defn}

\subsection{Formulation of the iterative argument}

As we described in the introduction, the proof of Theorem \ref{thm:thm1} is based on an iterative argument.  We now outline the main
setup of the iteration, and establish notation that will be used throughout the remainder of the paper.  We begin by fixing $\gamma\in (0,1)$ and
$\beta<\frac{1}{3}$ with $\beta+2\gamma<1$. 

For any natural number $q\in\mathbb{N}_{0}$, we set 
\begin{align}
\lambda_{q} & :=\left\lceil a^{(b^{q})}\right\rceil \\
\delta_{q} & :=\lambda_{q}^{-2\beta}
\end{align}
 with $a\gg1$, $0<b-1\ll1$ (to be chosen later). We remark that $\lambda_{q}$
will be the frequency parameter (made an integer for phase functions),
while $\delta_{q}$ will be the pointwise size of the Reynolds stress.

With $\alpha>0$ sufficiently small, and $\sigma>0$ (to be chosen
later), we set 
\begin{align}
\epsilon_{q} & :=\lambda_{q}^{-\sigma}\quad\\
\tau_{q} & :=C_{q}\delta_{q}^{-\frac{1}{2}}\lambda_{q}^{-1-3\alpha}\label{eq:tau_q}
\end{align}
where $C_{q}\sim1$ is an inessential constant such that $\epsilon_{q-1}\tau_{q-1}\tau_{q}^{-1}\in\mathbb{N}_{1}$
(for gluing purposes). For convenience, from this point on, we will
not write out $C_{q}$ explicitly.  The parameter $\tau_{q}$ will be the time of
local existence for regular solutions, while the quantity $\epsilon_{q}\tau_{q}$
will be the length of the overlapping region between two temporal
cutoffs.  

We now formulate the main inductive hypothesis on which the construction is based. Let $T\geq 1$ and $\nu \in (0,1]$ be arbitrary constants. For the first step of the 
induction, we pick any 
positive $\epsilon_{-1},\tau_{-1}$ such that $5\epsilon_{-1}\tau_{-1}=\frac{T}{3}$.  

For every $q\in\mathbb{N}_{0}$, we 
assume that there exist $v_{q}$ and $R_{q}$ smooth such that,
\begin{enumerate}
\item[(i)] $(v_q,R_q)$ solves the $(\nu,\gamma,T)$-fNSR equations in (\ref{eq:hypo_NSR}),
\item[(ii)] we have the estimates 
\begin{align}
\|v_{q}\|_{L^{\infty}} & \leq1-\delta_{q}^{1/2},\label{eq:vq_sup}\\
\|\grad v_{q}\|_{L^{\infty}} & \leq M\delta_{q}^{1/2}\lambda_{q},\label{eq:grad_vq_sup}\\
\|R_{q}\|_{L^{\infty}} & \leq\epsilon_{q}\delta_{q+1}\lambda_{q}^{-3\alpha},\label{eq:Rq_sup}
\end{align}
where $M$ is a universal geometric constant (depending on $d$), and
\item[(iii)] letting $\mathcal{B}_{q}=\bigcup_{i}I_{i}^{b,q}$ denote the current ``bad''
set consisting of disjoint closed intervals of length $5\epsilon_{q-1}\tau_{q-1}$,
and letting $$\mathcal{G}_{q}=[0,T]\setminus\mathcal{B}_{q}=\bigcup_{i}I_{i}^{g,q}$$
denote the current ``good'' set consisting of disjoint open intervals,
we have 
\begin{equation}
R_{q}|_{\mathcal{G}_{q}+B(0,\epsilon_{q-1}\tau_{q-1})}\equiv0,\label{eq:Rq_zerowhere}
\end{equation}
where $\mathcal{G}_{q}+B(0,\epsilon_{q-1}\tau_{q-1})$ denotes the
$\epsilon_{q-1}\tau_{q-1}$-neighborhood of $\mathcal{G}_{q}$, and within
this neighborhood we have the improved bounds 
\begin{align}
\|v_{q}\|_{N+1} & \lesssim_{N,\neg q}\delta_{q-1}^{1/2}\lambda_{q-1}\ell_{q-1}^{-N},\quad\textrm{for all }N\geq0.\label{eq:goodbounds}
\end{align}
\end{enumerate}

We note the presence of $\epsilon_{q}$ in (\ref{eq:Rq_sup}), which serves to compensate
for the sharp time cutoffs in our gluing construction. 

The main iterative proposition is given by the following statement.

\begin{prop}[Iteration for the $(\nu,\gamma,T)$-fNSR equations]
\label{prop:iterative}
We fix 
\begin{gather}
0<b-1  \ll_{\beta,\gamma}1,\label{eq:smallness_b}\\
0<\sigma  <\frac{(b-1)(1-\beta-2b\beta)}{b+1},\label{eq:sigma_bound}\\
0<\alpha  \ll_{\sigma,b,\beta,\gamma}1,\label{eq:alpha_bound}\\
a  \gg_{\alpha,\sigma,b,\beta,\gamma}1,\nonumber 
\end{gather}
and suppose that $v_{q}$ and $R_{q}$ are smooth functions which satisfy the properties (i)--(iii) above.  Then there
exist $v_{q+1}$ and $R_{q+1}$ satisfying those same properties but
with $q$ replaced by $q+1$. Moreover, we have 
\begin{align}
\|v_{q}-v_{q+1}\|_{0}+\lambda{}_{q+1}^{-1}\|v_{q}-v_{q+1}\|_{1} & \leq M\delta_{q+1}^{1/2}\label{eq:iter_prop_est}
\end{align}
and $v_{q+1}=v_{q}$ on $\mathcal{G}_{q}\times\mathbb{T}^{d}$, $\mathcal{G}_{q}\subset\mathcal{G}_{q+1},\left|\mathcal{B}_{q+1}\right|\leq\epsilon_{q}\left|\mathcal{B}_{q}\right|$.
\end{prop}

% Preview source code for paragraph 57

\begin{rem}
\label{rem:crucial_rem}We crucially remark that the parameters $b,\sigma,\alpha$,
and $a$ only depend on $\beta,\gamma$ and $d$. In particular, they
do not depend on $q,T$ or $\nu$ (as long as $\nu\leq1$ and $T\geq1$).
\end{rem}

The proof of \Propref{iterative} is given in \Secref{Proof-of-iteration} below. In the remainder of this section, we use this result
to prove Theorem \ref{thm:thm1}.

\begin{proof}[Proof of Theorem \ref{thm:thm1} via \Propref{iterative}]
Let $\eta$ be a smooth temporal cutoff on $[0,T]$ such that $\mathbf{1}_{[0,\frac{2}{5}T]}\geq\eta\geq\mathbf{1}_{[0,\frac{1}{3}T]}$, and set
\[
v_{0}=\eta V_{1}+\left(1-\eta\right)V_{2}.
\]

Since the dissipative terms are linear and $\int_{\mathbb{T}^{d}}\left(V_{1}-V_{2}\right)=0$,
if we set
\[
R_{0}=\partial_{t}\eta\mathcal{R}\left(V_{1}-V_{2}\right)-\eta\left(1-\eta\right)\left(V_{1}-V_{2}\right)\otimes\left(V_{1}-V_{2}\right)
\]
where $\mathcal{R}$ is the antidivergence operator defined in \Appref{Geometric-preliminaries}, then $\left(v_{0},R_{0}\right)$ solves the $(1,\gamma,T)$-fNSR equations from (\ref{eq:hypo_NSR}).

We now aim to apply \Propref{iterative}.  To do this, we rescale in time by a positive parameter $\zeta$, i.e.
\[
v_{0}^{\zeta}\left(t,x\right)=\zeta v_{0}\left(\zeta t,x\right),\quad R_{0}^{\zeta}=\zeta^{2}R_{0}\left(\zeta t,x\right)
\]
Then $(v_0^\zeta,R_0^\zeta)$ solves the $(\zeta,\gamma,\zeta^{-1}T)$-fNSR equations. 

We now recall that we are allowed to make $\zeta$ arbitrarily small because of Remark \ref{rem:crucial_rem}. For $\zeta=\zeta(T,V_{1},V_{2},a,b,\alpha,\sigma,\beta)$ small enough, the conditions
(\ref{eq:vq_sup})-(\ref{eq:Rq_sup}) of (ii) in the inductive hypothesis for \Propref{iterative} are satisfied for the case $q=0$, and we also have $$\zeta^{-1}T > 1 > \zeta.$$ In addition, (iii) is satisfied by letting $\mathcal{B}_{0}^{\zeta}=\left[\frac{T}{3\zeta},\frac{2T}{3\zeta}\right]$. 

Repeatedly applying \Propref{iterative} for the $(\zeta,\gamma,\zeta^{-1} T)$-fNSR equations, we get a sequence $\left(v_{q}^{\zeta},R_{q}^{\zeta},\mathcal{B}_{q}^{\zeta}\right)$
such that
\renewcommand{\labelenumi}{\roman{enumi}.}
\begin{enumerate}
\item[(a)] $\left(v_{q}^{\zeta}\right)_{q\in\mathbb{N}_{0}}$ converges in $C_{t}^{0}C_{x}^{\beta-}$
to some $v^{\zeta}$.
\item[(b)] $\lVert R_{q}^{\zeta}\rVert_{C_{t,x}^0}\rightarrow 0$ as $q\rightarrow\infty$, and
\item[(c)] $\mathcal{B}_{q+1}^{\zeta}\subset\mathcal{B}_{q}^{\zeta}$ and $v^{\zeta}=v_{q}^{\zeta}$
on $\mathcal{G}_{q}^{\zeta}\times\mathbb{T}^{d}$.
\end{enumerate}
\renewcommand{\labelenumi}{\arabic{enumi}.}

As a consequence of (c), $v^\zeta$ is smooth on each set $\mathcal{G}_q^\zeta \times\mathbb{T}^d$.  Moreover, $v^{\zeta}$ is a weak solution of the $(\zeta,\gamma,\zeta^{-1} T)$-fNS equations.

To conclude, we note that the transformations
\begin{align*}
v_{q}\left(t,x\right) & :=\zeta^{-1}v_{q}^{\zeta}\left(\zeta^{-1}t,x\right)\\
v\left(t,x\right) & :=\zeta^{-1}v^{\zeta}\left(\zeta^{-1}t,x\right)\\
\mathcal{B}_{q} & :=\zeta^{-1}\mathcal{B}_{q}^{\zeta}
\end{align*}
invert the time-rescaling.  The bad set is then 
\[
\mathcal{B}:=\bigcap_{q}\mathcal{B}_{q}.
\]
Moreover, noting that the choice of $V_2$ was arbitrary, the solution $v$ is nonunique.

We now verify that $\mathcal{B}$ has the desired Hausdorff dimension.  Note that the set $\mathcal{B}_{q}$ consists of $\sim\tau_{q}^{-1}\prod_{i=1}^{q-1}\epsilon_{i}$
intervals of length $\sim_{\zeta}\epsilon_{q}\tau_{q}$. It therefore
follows that 
\begin{align*}
\dim_{\mathrm{Hausdorff}}(\mathcal{B}) \leq\mathrm{dim}_{\mathrm{box}}(\mathcal{B})&\leq\lim_{q\to\infty}\frac{\ln\Big(\tau_{q}^{-1}\prod_{i=1}^{q-1}\epsilon_{i}\Big)}{\ln\left(\epsilon_{q}^{-1}\tau_{q}^{-1}\right)}\\
&=\lim_{q\to\infty}\frac{\ln(a)b^{q}\left(1+3\alpha-\beta-\frac{\sigma}{b-1}\right)}{\ln(a)b^{q}\left(1+3\alpha+\sigma-\beta\right)}\\
 & =1-\frac{\sigma b}{\left(b-1\right)\left(1+3\alpha+\sigma-\beta\right).}
\end{align*}
Choosing $\alpha$ sufficiently small, and then choosing $\sigma$ sufficiently close to $\frac{(b-1)\left(1-\beta-2b\beta\right)}{b+1}$
and $b>1$ sufficiently close to $1$, we get the bound
\[
\dim(\mathcal{B})\leq\left(\frac{1+\beta}{2(1-\beta)}\right)^{+}
\]
 as desired.

It remains to choose $\beta_{1}$ to ensure that the solution
lies in $C_{t}^{0}C_{x}^{\beta-}\cap L_{t}^{1}C_{x}^{\beta_{1}}.$ For this, note that 
since $\left|\mathcal{B}_{q+1}\right|\lesssim_{\zeta}\prod_{i=1}^{q}\epsilon_{i}$,
we have 
\[
\left\Vert v_{q+1}-v_{q}\right\Vert _{L_{t}^{1}C_{x}^{\beta_{1}}}\lesssim_{\zeta}\delta_{q+1}^{1/2}\lambda_{q+1}^{\beta_{1}}\left(\prod_{i=1}^{q}\epsilon_{i}\right)
\]
The right-hand side is then summable in $q$, provided that 
\[
-\beta+\beta_{1}-\frac{\sigma}{b-1}<0.
\]
We may therefore choose $\beta_{1}<\beta+\frac{\sigma}{b-1}<\frac{1-b\beta}{b+1}<\frac{1-\beta}{2}$, which completes the proof.
\end{proof}

\section{\label{sec:Proof-of-iteration}Proof of \Propref{iterative}}

In this section, we give the proof of the main iterative result, \Propref{iterative}, which was used to prove Theorem \ref{thm:thm1} in the
previous section.  As we described in the introduction, the argument makes use of three steps -- a mollification procedure, a gluing construction, and a
perturbation result arising from convex integration.  To simplify the exposition, we discuss each step below, and after isolating a
few technical lemmas whose proofs are deferred to \Secref{Proof-of-gluing} and \Secref{Convex-integration-and}, we give the proof of \Propref{iterative}.

We define the length scale of mollification
\begin{align}
\ell_{q} & :=\frac{\delta_{q+1}^{\frac{1}{2}}}{\lambda_{q}^{1+\frac{\sigma}{2}+\frac{3\alpha}{2}}\delta_{q}^{\frac{1}{2}}}.\label{eq:length_scale}
\end{align}
To simplify notation, we will often abbreviate $\ell_{q}$ as $\ell$ (unless otherwise indicated).

For technical convenience, we record several useful parameter inequalities.  The first set of these are essential conversions,
\begin{gather}
\epsilon_{q}^{\frac{1}{2}}\tau_{q}\delta_{q+1}^{\frac{1}{2}}\ell_{q}^{-1}  \ll1\label{eq:time-length}\\
\lambda_{q}\ll\ell_{q}^{-1}  \ll\lambda_{q+1}\ll\lambda_{q}^{\frac{3}{2}}\label{eq:length_Freq}\\
\delta_{q-1}^{\frac{1}{2}}\lambda_{q-1}\ell_{q}^{2-2\alpha}  \ll\epsilon_{q}\tau_{q}\delta_{q+1}.\label{eq:good_bad1}
\end{gather}
Indeed, the bound (\ref{eq:time-length}) comes from $\alpha>0$, while
the bound $\ell_{q}^{-1}\ll\lambda_{q+1}$ in (\ref{eq:length_Freq}) follows by recalling that
that $\alpha$ can be made arbitrarily small by (\ref{eq:alpha_bound}), so that
(\ref{eq:sigma_bound}) implies $\sigma<2\left(b-1\right)\left(1-\beta\right)$, and thus
\begin{equation}
-\beta b+\beta-1-\frac{\sigma}{2}+b>0.\label{eq:intermed_para}
\end{equation}
Similarly, by neglecting $\alpha$, (\ref{eq:good_bad1}) comes from
$\left(b-1\right)\left(1-\beta\right)>0$, which is obvious.

In order to partition the time intervals for gluing, we also need
the bound 
\begin{equation}
\epsilon_{q}\tau_{q}\ll\tau_{q}\ll\epsilon_{q-1}\tau_{q-1}\label{eq:time_interval}
\end{equation}
which comes from the inequality $\sigma<\left(1-\beta\right)(b-1)$,
a consequence of (\ref{eq:sigma_bound}). We also have the special
case 
\[
\tau_{0}\ll\frac{1}{15}\leq\frac{T}{15}=\epsilon_{-1}\tau_{-1}
\]
because $T\geq1$. This allows $a,b,\beta,\alpha$
to be independent of $T$, and we use this crucial fact in the proof of Theorem \ref{thm:thm1}. 

To control the dissipative term in the gluing construction, we will also find it useful to observe the bound
\begin{align}
\tau_{q}\ell_{q}^{-\alpha-2\gamma} & \lesssim1,\label{eq:diss_1}
\end{align}
which, since $\alpha$ is negligible by (\ref{eq:alpha_bound}), 
comes from the inequality $\sigma<\left(1-\beta\right)\left(\frac{1-2\gamma}{\gamma}\right)-2b\beta$.
Because of (\ref{eq:smallness_b}) and $\beta+2\gamma<1$, this is implied
by (\ref{eq:sigma_bound}). 

Next, to control the stress size for the induction step, we note that
\begin{align}
\epsilon_{q}^{-1}\delta_{q+1}^{\frac{1}{2}}\delta_{q}^{\frac{1}{2}}\lambda_{q+1}^{-1+10\alpha}\lambda_{q}^{1+10\alpha} & \lesssim\epsilon_{q+1}\delta_{q+2}\lambda_{q+1}^{-4\alpha}\label{eq:stress_size_ind1},
\end{align}
which, after neglecting $\alpha$, comes from
\[
-b\beta-\beta-b+1+\sigma\leq-b\sigma-b^{2}\left(2\beta\right),
\]
which is precisely (\ref{eq:sigma_bound}).

Lastly, for the dissipative error in the final stress, we observe that
\begin{align}
\delta_{q+1}^{1/2}\lambda_{q+1}^{-1+2\gamma+10\alpha} & \lesssim\epsilon_{q+1}\delta_{q+2}\lambda_{q+1}^{-4\alpha},\label{eq:stress_size_ind3}
\end{align}
which comes from 
\[
-\beta-1+2\gamma<b\left(-2\beta\right)-\sigma
\]
which in view of (\ref{eq:smallness_b}) and $\beta+2\gamma<1$, is a consequence of 
(\ref{eq:sigma_bound}).

\subsection{The mollification step}

With $\ell$ as defined in (\ref{eq:length_scale}) and $\psi_{\ell}$
a smooth standard radial mollifier in space of length $\ell$, we set
\begin{align*}
v_{\ell} & :=\psi_{\ell}*v_{q}.
\end{align*}

By standard mollification estimates and (\ref{eq:grad_vq_sup}) we
have 
\begin{align}
\left\Vert v_{\ell}-v_{q}\right\Vert _{0} & \lesssim\delta_{q}^{1/2}\lambda_{q}\ell=\epsilon_{q}^{\frac{1}{2}}\delta_{q+1}^{\frac{1}{2}}\lambda_{q}^{-\frac{3\alpha}{2}}\label{eq:lq_0}\\
\|\grad^{N}v_{\ell}\|_{L^{\infty}} & \lesssim_{N}\delta_{q}^{1/2}\lambda_{q}\ell^{-N+1}\label{eq:grad_v_l}
\end{align}
for any $N\in\mathbb{N}_{1}$. Moreover, by setting
\begin{align*}
R_{\ell} & :=\psi_{\ell}*R_{q}+v_{\ell}\otimes v_{\ell}-\psi_{\ell}*(v_{q}\otimes v_{q})
\end{align*}
the pair $(v_{\ell},R_{\ell})$ solves the $(\nu,\gamma,T)$-fNSR equations.

Moreover, by using (\ref{eq:vq_sup}), (\ref{eq:grad_vq_sup}), (\ref{eq:Rq_sup}), (\ref{eq:length_scale}), (\ref{eq:length_Freq}) and the 
usual commutator estimate
\[
\left\Vert \left(f*\psi_{l}\right)\left(g*\psi_{l}\right)-\left(fg\right)*\psi_{l}\right\Vert _{C^{r}}\lesssim_{r}l^{2-r}\left\Vert f\right\Vert _{C^{1}}\left\Vert g\right\Vert _{C^{1}}
\]
for $f,g\in C^{\infty}\left(\mathbb{T}^{d}\right)$ and $l>0,r\geq0$ (see, e.g. \cite[Proposition A.2]{buckmasterOnsagerConjectureAdmissible2017}), we obtain
\begin{align}
\|R_{\ell}\|_{N+\alpha} & \lesssim_{N}\ell^{-N-\alpha}\left\Vert R_{q}\right\Vert _{C^{0}}+\ell^{2-N-\alpha}\left\Vert v_{q}\right\Vert _{C^{1}}^{2}\nonumber \\
 & \lesssim\ell^{-N-\alpha}\delta_{q+1}\epsilon_{q}\lambda_{q}^{-3\alpha}\lesssim\epsilon_{q}\delta_{q+1}\ell^{-N+\alpha}\label{eq:Rl_estimate}
\end{align}
for all $N\in\mathbb{N}_{0}$. 

\subsection{\label{subsec:The-gluing-step}The gluing step}

Recalling that $\tau_{q}$ was defined in (\ref{eq:tau_q}), we set $$t_{j}:=j\tau_{q}.$$

Let $\mathcal{J}$ be the set of indices $j$ such that 
\[
\left[t_{j}-2\epsilon_{q}\tau_{q},t_{j}+3\epsilon_{q}\tau_{q}\right]\subset B_{q}
\]
These are the ``bad'' indices that will be part of $\mathcal{B}_{q+1}$
and we have $\#(\mathcal{J})\sim\tau_{q}^{-1}\prod_{p=1}^{q-1}\epsilon_{p}$.

Then we define $\mathcal{J}^{*}=\{j\in\mathcal{J}|j+1\in\mathcal{J}\}$.
These are the indices where we will apply the following local wellposedness
result from \cite{derosaInfinitelyManyLeray2019}.

\begin{lem}[Proposition 3.5 in \cite{derosaInfinitelyManyLeray2019}]
\label{lem:local_existence}Given $\alpha\in\left(0,1\right)$, $\nu\in (0,1]$,  any
divergence-free vector field $u_{0}\in C^{\infty}(\mathbb T^d)$ and $T\lesssim_{\alpha}\left\Vert u_{0}\right\Vert _{1+\alpha}^{-1}$,
there exists a unique solution $u$ to the $(\nu,\gamma,T)$-fNS equations on $[0,T]\times\mathbb{T}^{d}$
such that $u\left(0,\cdot\right)=u_{0}$ and 
\[
\left\Vert u\right\Vert _{N+\alpha}\lesssim_{N,\alpha}\left\Vert u_{0}\right\Vert _{N+\alpha}\quad\textrm{for all}\quad N\in\mathbb{N}_{1}.
\]
\end{lem}

Using this lemma, for any $j\in\mathcal{J}^{*}$, we define $v_{j}$
to be the solution of the hypodissipative Navier-Stokes equations
\begin{align*}
\dd_{t}v_{j}+\nu (-\Delta)^{\gamma}v_{j}+\divop v_{j}\otimes v_{j}+\grad p_{j} & =0\\
\divop v_{j} & =0\\
v_{j}(t_{j}) & =v_{\ell}(t_{j})
\end{align*}
 on $[t_{j},t_{j+2}]\times\T$. This is possible as 
\begin{align}
\tau_{q} & \lesssim\frac{\ell^{2\alpha}}{\delta_{q}^{1/2}\lambda_{q}}\label{eq:tau_l}\\
 & \ll\frac{\ell^{\alpha}}{\delta_{q}^{1/2}\lambda_{q}}\nonumber \\
&\lesssim\frac{1}{\|v_{\ell}(t_{j})\|_{1+\alpha}},\nonumber 
\end{align}
where we have implicitly used (\ref{eq:length_Freq}) and (\ref{eq:grad_v_l}).

We then have the bounds
\begin{align}
\|v_{j}\|_{L_{t}^{\infty}C_{x}^{N+\alpha}([t_{j},t_{{j+2}}]\times\T)} & \lesssim_{N}\|v_{\ell}(t_{j})\|{}_{C_{x}^{N+\alpha}}\label{eq:badeulerbounds}\\
 & \lesssim_{N}\delta_{q}^{1/2}\lambda_{q}\ell^{-N+1-\alpha},
\end{align}
 for $N\in\mathbb{N}_{1}$.

\begin{figure}[h]
\centering{}\includegraphics[scale=0.16]{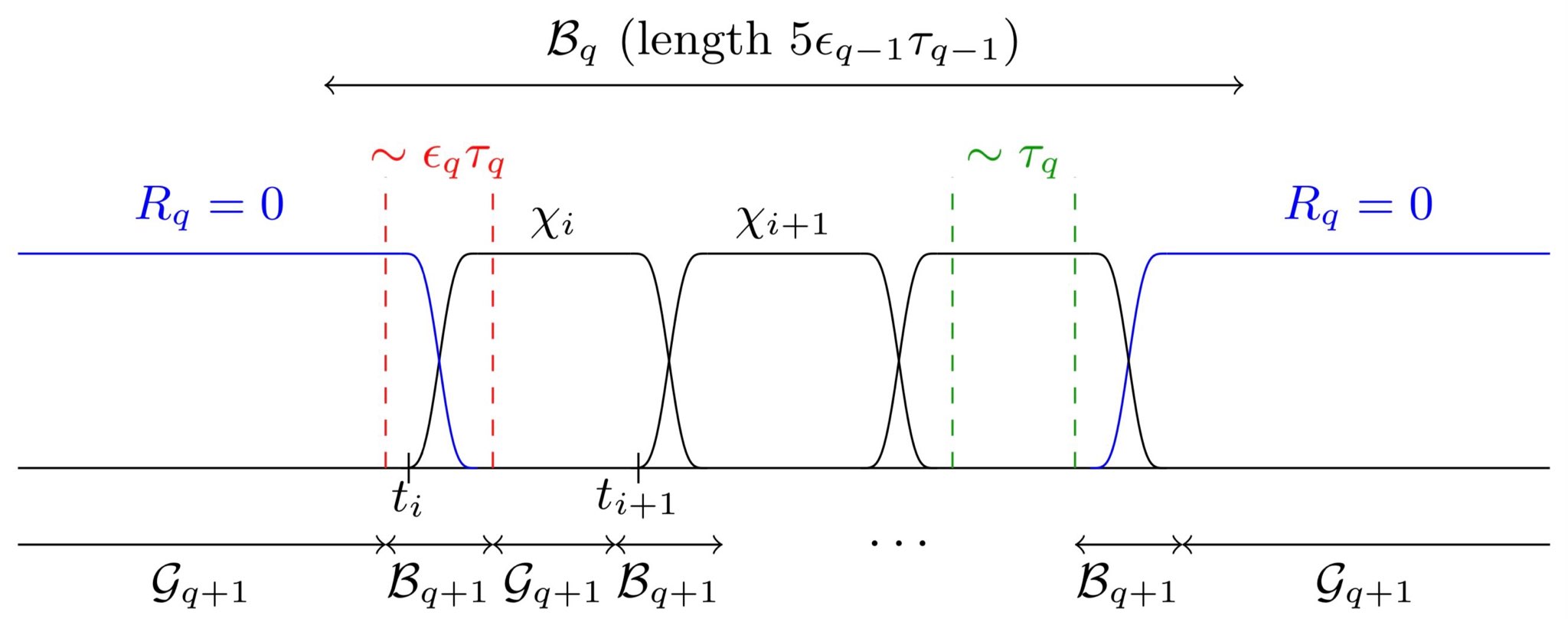}\caption{gluing scheme}
\end{figure}

Recall that $\mathcal{B}_{q}=\bigcup_{i}I_{i}^{b,q}$ is closed and
$\mathcal{G}_{q}=[0,T]\setminus\mathcal{B}_{q}=\bigcup_{i}I_{i}^{g,q}$
is open. Let $\{\chi_{j}^{b}\}_{j}\cup\{\chi_{i}^{g}\}_{i}$ be a
partition of unity of $[0,T]$ such that 
\begin{itemize}
\item $\supp\chi_{j}^{b}\subset[t_{j},t_{j+1}+\epsilon_{q}\tau_{q}]$ for
$j\in\mathcal{J}^{*}$, 
\item $\chi_{j}^{b}\equiv1$ in $[t_{j}+\epsilon_{q}\tau_{q},t_{j+1}]$
for $j\in\mathcal{J}^{*}$, 
\item $\supp\chi_{i}^{g}\subset I_{i}^{g,q}+B\left(0,\tau_{q}+\epsilon_{q}\tau_{q}\right)$, and
\item for $N\in\mathbb{N}_{0}$, 
\begin{equation}
\|\dd_{t}^{N}\chi_{i}^{g}\|_{L^{\infty}}+\|\dd_{t}^{N}\chi_{j}^{b}\|_{L^{\infty}}\lesssim_{N}(\epsilon_{q}\tau_{q})^{-N}.\label{eq:cutoff_time_deri}
\end{equation}
\end{itemize}

Note that because of (\ref{eq:Rq_zerowhere}) and (\ref{eq:time_interval}), we have $R_{q}=0$ on $\supp\chi_{i}^{g}$.

We now define the glued solution 
\begin{equation}
\overline{v}_{q}:=\sum_{i}\chi_{i}^{g}v_{q}+\sum_{j\in\mathcal{J^{*}}}\chi_{j}^{b}v_{j}.\label{eq:vbar_q}
\end{equation}
We also define $\mathcal{B}_{q+1}$ as the union of the intervals
$\left[t_{j}-2\epsilon_{q}\tau_{q},t_{j}+3\epsilon_{q}\tau_{q}\right]$
which lie in $\mathcal{B}_{q}$.

We will show in \Secref{Proof-of-gluing} that there exists a
smooth $\overline{R}_{q}$ such that $\left(\overline{v}_{q},\overline{R}_{q}\right)$
is a solution to (\ref{eq:hypo_NSR}).  For convenient notation, we define the material derivatives
\begin{align*}
D_{t,\ell} & :=\dd_{t}+v_{\ell}\cdot\grad\\
D_{t,q} & :=\dd_{t}+\overline{v}_{q}\cdot\grad
\end{align*}
We will then obtain the following estimates, which will be used to prove Proposition \ref{prop:iterative}.

\begin{prop}[Gluing estimates]
\label{prop:glue_est} For any $N\in\mathbb{N}_{0}$, we have 
\begin{align}
\left\Vert \overline{v}_{q}-v_{\ell}\right\Vert _{N+\alpha} & \lesssim_{N}\epsilon_{q}\tau_{q}\delta_{q+1}\ell^{-N-1+\alpha}\label{eq:glue1}\\
\left\Vert \overline{v}_{q}\right\Vert _{N+1} & \lesssim_{N}\delta_{q}^{\frac{1}{2}}\lambda_{q}\ell^{-N}\label{eq:glue2}\\
\|\overline{R}_{q}\|_{N+\alpha} & \lesssim_{N}\delta_{q+1}\ell^{-N+\alpha}\label{eq:glue3}\\
\|D_{t,q}\overline{R}_{q}\|_{N+\alpha} & \lesssim_{N}(\epsilon_{q}\tau_{q})^{-1}\delta_{q+1}\ell^{-N+\alpha}\label{eq:glue4}
\end{align}
\end{prop}

We will prove \Propref{glue_est} in \Secref{Proof-of-gluing} below.

We remark that the estimate (\ref{eq:glue1}) of \Propref{glue_est}, when combined with (\ref{eq:time-length}), implies in particular
\begin{equation}
\left\Vert \overline{v}_{q}-v_{\ell}\right\Vert _{\alpha}\lesssim\epsilon_{q}\tau_{q}\delta_{q+1}\ell^{-1+\alpha}\lesssim\delta_{q+1}^{\frac{1}{2}}\ell^{\alpha}\label{eq:glue5}.
\end{equation}

We also note that, because future modifications of the solution from this point on will
only happen in the temporal regions $\left[t_{j}-\epsilon_{q}\tau_{q},t_{j}+2\epsilon_{q}\tau_{q}\right]$
(where $j\in\mathcal{J}$), we will later have $v_{q+1}=\overline{v}_{q}$
and $\overline{R}_{q}=0$ outside those temporal regions. Furthermore,
(\ref{eq:Rq_zerowhere}) and (\ref{eq:goodbounds}) will hold with
$q$ changed to $q+1$.

\subsection{Perturbation step}

The third key step in the proof of Proposition \ref{prop:iterative} is a perturbation lemma arising from the convex
integration framework.  We state this result in the next proposition.

\begin{prop}[Convex integration]
\label{prop:convex_int} There is a smooth solution $\left(v_{q+1},R_{q+1}\right)$
to (\ref{eq:hypo_NSR}) which satisfies $v_{q+1}=\overline{v}_{q}$ outside
the temporal regions $\left[t_{j}-\epsilon_{q}\tau_{q},t_{j}+2\epsilon_{q}\tau_{q}\right]$
($j\in\mathcal{J}$), along with the estimates
\begin{align}
\left\Vert v_{q+1}-\overline{v}_{q}\right\Vert _{0}+\frac{1}{\lambda_{q+1}}\left\Vert v_{q+1}-\overline{v}_{q}\right\Vert _{1} & \leq\frac{M}{2}\delta_{q+1}^{\frac{1}{2}},\label{eq:perturb_1}
\end{align}
and
\begin{align}
\left\Vert R_{q+1}\right\Vert _{0} & \lesssim\epsilon_{q+1}\delta_{q+2}\lambda_{q+1}^{-4\alpha},\label{eq:finalR_=00007Bq+1=00007D}
\end{align}
where $M>0$ is a universal geometric constant (depending on $d$).
\end{prop}

The proof of this proposition will be given in \Secref{Convex-integration-and} below.

\subsection{Proof of the main iterative proposition}

With the above tools in hand, we are now ready to prove Proposition \ref{prop:iterative}, making use of Proposition \ref{prop:glue_est} and Proposition \ref{prop:convex_int}, which are proved in Sections \ref{sec:Proof-of-gluing} and \ref{sec:Convex-integration-and}, respectively.

\begin{proof}[Proof of Proposition \ref{prop:iterative}]
We first observe that
\begin{align*}
\|v_{q}-v_{q+1}\|_{0} & \leq\|v_{q}-v_{\ell}\|_{0}+\|v_{\ell}-\overline{v}_{q}\|_{0}+\|\overline{v}_{q}-v_{q+1}\|_{0}\\
 & \leq C\epsilon_{q}^{\frac{1}{2}}\delta_{q+1}^{\frac{1}{2}}\lambda_{q}^{-\frac{3\alpha}{2}}+C\delta_{q+1}^{\frac{1}{2}}\ell^{\alpha}+\frac{M}{2}\delta_{q+1}^{\frac{1}{2}}\leq M\delta_{q+1}^{\frac{1}{2}}
\end{align*}
where $C$ is shorthand for the implied constants of (\ref{eq:lq_0})
and (\ref{eq:glue5}). Since $$\max\{\epsilon_{q}^{\frac{1}{2}}\lambda_{q}^{-\frac{3\alpha}{2}},\ell^{\alpha}\}\rightarrow 0$$ as $a\rightarrow\infty$, 
the last inequality is true provided that $a$ is chosen sufficiently large. 

Similarly, for large $a$, because of (\ref{eq:grad_vq_sup}), (\ref{eq:glue2})
and (\ref{eq:perturb_1}), we have
\begin{align*}
\|v_{q}-v_{q+1}\|_{1} & \leq\underbrace{\left\Vert v_{q}\right\Vert _{1}+\|\overline{v}_{q}\|_{1}}_{C\delta_{q}^{\frac{1}{2}}\lambda_{q}}+\underbrace{\|\overline{v}_{q}-v_{q+1}\|_{1}}_{\frac{M}{2}\delta_{q+1}^{\frac{1}{2}}\lambda_{q+1}}\leq M\delta_{q+1}^{\frac{1}{2}}\lambda_{q+1}.
\end{align*}
We have thus shown (\ref{eq:iter_prop_est}), which in turn implies
(\ref{eq:vq_sup}) and (\ref{eq:grad_vq_sup}) with $q$ replaced
by $q+1$. On the other hand, (\ref{eq:finalR_=00007Bq+1=00007D})
yields the next iteration of (\ref{eq:Rq_sup}) (for large enough
$a$). Recalling that all the desired properties regarding $\mathcal{B}_{q+1}$ were
established in \Subsecref{The-gluing-step}, this completes the proof of the proposition.
\end{proof}

\section{\label{sec:Proof-of-gluing}Gluing estimates}

In this section, we construct $\overline{R}_q$ and prove the gluing estimate results in \Propref{glue_est}, which played a key role in the proof of Proposition \ref{prop:iterative} in the previous section. 

We recall that $\overline{v}_{q}$ was defined in (\ref{eq:vbar_q}). We first note
that (\ref{eq:glue2}) follows immediately from (\ref{eq:badeulerbounds})
and (\ref{eq:goodbounds}). On the other hand, (\ref{eq:glue3}) and
(\ref{eq:glue4}) hold automatically outside the overlapping temporal
regions $\left[t_{j},t_{j}+\epsilon_{q}\tau_{q}\right]$ (where $j\in\mathcal{J}$),
since $\overline{v}_{q}$ is an exact solution and the stress
is therefore zero in this regime. We now consider what happens near the overlapping 
regions. 

\subsection{Bad-bad interface}

Consider any index $j\in\mathcal{J}^{*}$ such that $j+1\in\mathcal{J}^{*}$.
Then $\supp(\chi_{j}^{b}\chi_{j+1}^{b})$ lies in an interval of length
$\epsilon_{q}\tau_{q}$ where $\overline{v}_{q}$ satisfies 
\[
\dd_{t}\overline{v}_{q}+\nu(-\Delta)^{\gamma}\overline{v}_{q}+\divop\overline{v}_{q}\otimes\overline{v}_{q}+\grad\overline{p}_{q}=\divop\overline{R}_{q},
\]
where 
\begin{equation}
\overline{R}_{q}=\dd_{t}\chi_{j}^{b}\mathcal{R}(v_{j}-v_{j+1})-\chi_{j}^{b}(1-\chi_{j}^{b})(v_{j}-v_{j+1})\otimes(v_{j}-v_{j+1}).\label{eq:glued_stress}
\end{equation}
and $\mathcal{R}$ is as defined in \Appref{Geometric-preliminaries}.

To treat the fractional Laplacian term, we recall the following lemma from \cite{derosaInfinitelyManyLeray2019}.
\begin{lem}[Theorem B.1 in \cite{derosaInfinitelyManyLeray2019}]
For any $\gamma,\epsilon>0$ and $\beta\geq0$ such that $\beta+2\gamma+\epsilon\leq1$,
we have 
\[
\left\Vert \left(-\Delta\right)^{\gamma}f\right\Vert _{\beta}\lesssim_{\epsilon}\left\Vert f\right\Vert _{\beta+2\gamma+\epsilon}\quad\forall f\in C^{\beta+2\gamma+\epsilon}.
\]
\end{lem}

As usual, we decompose $v_{j}-v_{j+1}=(v_{j}-v_{\ell})-(v_{j+1}-v_{\ell})$.
By symmetry, we only need to prove estimates for $v_{j}-v_{\ell}$.
\begin{prop}
\label{prop:vglue}For $N\in\mathbb{N}_{0}$ and $t\in\left(t_{j},t_{j}+2\tau_{q}\right)$, we have
\begin{align}
\|v_{j}-v_{\ell}\|_{N+\alpha} & \lesssim_{N}\epsilon_{q}\tau_{q}\delta_{q+1}\ell^{-N-1+\alpha}\label{eq:vestimate-1}\\
\|(\dd_{t}+v_{\ell}\cdot\grad+\nu\left(-\Delta\right)^{\gamma})\left(v_{j}-v_{\ell}\right)\|_{N+\alpha} & \lesssim_{N}\epsilon_{q}\delta_{q+1}\ell^{-N-1+\alpha}\label{eq:vlestimate}\\
\|D_{t,\ell}\left(v_{j}-v_{\ell}\right)\|_{N+\alpha} & \lesssim_{N}\epsilon_{q}\delta_{q+1}\ell^{-N-1+\alpha}\label{eq:vtransportestimate-1}
\end{align}
\end{prop}

\begin{proof}
We observe that
\begin{align}
\left(\partial_{t}+v_{\ell}\cdot\nabla+\nu\left(-\Delta\right)^{\gamma}\right)\left(v_{\ell}-v_{j}\right) & =-\left(v_{\ell}-v_{j}\right)\cdot\nabla v_{j}-\nabla\left(p_{\ell}-p_{j}\right)+\Div R_{\ell},\label{eq:vl_vj_subtract}
\end{align}
and
\begin{align}
\nabla\left(p_{\ell}-p_{j}\right) & =\mathcal{P}_{1}\left(-\left(v_{\ell}-v_{j}\right)\cdot\nabla v_{\ell}-\left(v_{\ell}-v_{j}\right)\cdot\nabla v_{j}+\Div R_{\ell}\right),\label{eq:vl_vj_subtract2}
\end{align}
where $\mathcal{P}_{1}$ is as defined in \Appref{Geometric-preliminaries}, and (\ref{eq:ident_P1}) was implicitly used.

Then, as usual, (\ref{eq:vestimate-1}) and (\ref{eq:vlestimate})
follow from Gronwall and modified transport estimates exactly as in
\cite[Proposition 5.3]{derosaInfinitelyManyLeray2019} (which
in turn mirrors \cite[Proposition 3.3]{buckmasterOnsagerConjectureAdmissible2017}).

To derive (\ref{eq:vtransportestimate-1}) from (\ref{eq:vlestimate}),
we observe that
\begin{align}
\|(-\Delta)^{\gamma}(v_{j}-v_{\ell})\|_{N+\alpha} & \lesssim\|v_{j}-v_{\ell}\|_{N+2\alpha+2\gamma}\nonumber \\
 & \lesssim\epsilon_{q}\tau_{q}\delta_{q+1}\ell^{-1-2\gamma-N}\nonumber \\
 & \lesssim\epsilon_{q}\delta_{q+1}\ell^{-N-1+\alpha}\label{eq:new_in_glue}
\end{align}
 where the last inequality comes from (\ref{eq:diss_1}).
\end{proof}
We have proven (\ref{eq:glue1}) for any $t\in\left(t_{j},t_{j}+2\tau_{q}\right)$.

Now we define the potentials $z_{j}:=\mathcal{B}v_{j}$, $z_{\ell}:=\mathcal{B}v_{\ell}$,
where $\mathcal{B}$ is as defined in \Appref{Geometric-preliminaries}.
\begin{prop}
\label{prop:zglue}For $N\in\mathbb{N}_{0}$ and $t\in\left(t_{j},t_{j}+2\tau_{q}\right)$:
\begin{align}
\|z_{j}-z_{\ell}\|_{N+\alpha} & \lesssim_{N}\epsilon_{q}\tau_{q}\delta_{q+1}\ell^{-N+\alpha}\label{eq:zestimate}\\
\|(\dd_{t}+v_{\ell}\cdot\grad+\nu\left(-\Delta\right)^{\gamma})\left(z_{j}-z_{\ell}\right)\|_{N+\alpha} & \lesssim_{N}\epsilon_{q}\delta_{q+1}\ell^{-N+\alpha}\label{eq:ztransport_lambda_est}\\
\|D_{t,\ell}\left(z_{j}-z_{\ell}\right)\|_{N+\alpha} & \lesssim_{N}\epsilon_{q}\delta_{q+1}\ell^{-N+\alpha}\label{eq:ztransportestimate}
\end{align}
\end{prop}

\begin{proof}
First, we note that for any divergence-free vector field $X$ and
2-form $\omega$, we have 
\begin{align*}
X^{i}\partial_{i}\partial^{j}\omega_{jk} & =\partial^{j}\left(X^{i}\partial_{i}w_{jk}\right)-\partial_{i}\left(\partial^{j}X^{i}\omega_{jk}\right)\\
\left(\partial^{j}\omega_{jk}\right)\partial^{k}X^{i} & =\partial^{j}\left(\omega_{jk}\partial^{k}X^{i}\right)\\
\left[d,\nabla_{X}\right]\omega & =dx^{i}\wedge\partial_{i}\nabla_{X}\omega-\nabla_{X}\left(dx^{i}\wedge\partial_{i}\omega\right)\\
 & =dx^{i}\wedge\left(\partial_{i}X^{j}\right)\left(\partial_{j}\omega\right)=\partial_{j}\left(\left(\partial_{i}X^{j}\right)dx^{i}\wedge\omega\right).
\end{align*}
Because we only care about estimates instead of how the indices contract,
we can write in schematic notation (neglecting indices and linear
combinations): 
\begin{align*}
\nabla_{X}\delta\omega & =\delta\left(\nabla_{X}\omega\right)+\Div\left(\nabla X*\omega\right)\\
\left(\delta\omega\right)\cdot\nabla X & =\Div\left(\nabla X*\omega\right)\\
\left[d,\nabla_{X}\right]\omega & =\Div\left(\nabla X*\omega\right)
\end{align*}

Define $\widetilde{z}:=z_{\ell}-z_{j}$. Then we have $d\widetilde{z}=0$
and $\sharp\delta\widetilde{z}=v_{\ell}-v_{j}$. From (\ref{eq:vl_vj_subtract})
and the schematic identities above, we have,
\begin{align*}
\delta\left(\partial_{t}\widetilde{z}+\nabla_{v_{\ell}}\widetilde{z}+\nu \left(-\Delta\right)^{\gamma}\widetilde{z}\right) & =\Div\left(\nabla v_{j,\ell}*\widetilde{z}\right)-d\left(p_{\ell}-p_{j}\right)+\Div R_{\ell}\\
d\left(\partial_{t}\widetilde{z}+\nabla_{v_{\ell}}\widetilde{z}+\nu  \left(-\Delta\right)^{\gamma}\widetilde{z}\right) & =\Div\left(\nabla v_{\ell}*\widetilde{z}\right),
\end{align*}
and thus 
\begin{align*}
\partial_{t}\widetilde{z}+\nabla_{v_{\ell}}\widetilde{z}+\nu \left(-\Delta\right)^{\gamma}\widetilde{z} & =\left(-\Delta\right)^{-1}d\circ\Div\left(\nabla v_{j,\ell}*\widetilde{z}+R_{\ell}\right)+\left(-\Delta\right)^{-1}\delta\circ\Div\left(\nabla v_{\ell}*\widetilde{z}\right),
\end{align*}
where $v_{j,\ell}$ could be $v_{j}$ or $v_{\ell}$ (they obey the
same estimates by \Lemref{local_existence}). As $\left(-\Delta\right)^{-1}d\circ\Div$
and $\left(-\Delta\right)^{-1}\delta\circ\Div$ are Calderón-Zygmund
operators, we have 
\begin{align}
&\left\Vert \left(D_{t,\ell}+\nu\left(-\Delta\right)^{\gamma}\right)\widetilde{z}\right\Vert _{N+\alpha}\nonumber \\
&\hspace{0.2in} \lesssim\left\Vert \nabla v_{j,\ell}\right\Vert _{N+\alpha}\left\Vert \widetilde{z}\right\Vert _{\alpha}+\left\Vert \nabla v_{j,\ell}\right\Vert _{\alpha}\left\Vert \widetilde{z}\right\Vert _{N+\alpha}+\left\Vert R_{\ell}\right\Vert _{N+\alpha}\nonumber \\
&\hspace{0.2in} \lesssim\ell^{-N-\alpha}\lambda_{q}\delta_{q}^{\frac{1}{2}}\left\Vert \widetilde{z}\right\Vert _{\alpha}+\ell^{-\alpha}\lambda_{q}\delta_{q}^{\frac{1}{2}}\left\Vert \widetilde{z}\right\Vert _{N+\alpha}+\ell^{-N+\alpha}\epsilon_{q}\delta_{q+1}\nonumber \\
&\hspace{0.2in} \lesssim\ell^{-N+\alpha}\tau_{q}^{-1}\left\Vert \widetilde{z}\right\Vert _{\alpha}+\ell^{\alpha}\tau_{q}^{-1}\left\Vert \widetilde{z}\right\Vert _{N+\alpha}+\ell^{-N+\alpha}\epsilon_{q}\delta_{q+1}\label{eq:Lztilde}
\end{align}
where we have used (\ref{eq:tau_l}) to pass to the last line.

By the modified transport estimate in \cite[Proposition 3.3]{derosaInfinitelyManyLeray2019},
we also have
\begin{align}
\left\Vert \widetilde{z}\left(t\right)\right\Vert _{\alpha} & \lesssim\int_{t_{j}}^{t}\left\Vert \left(D_{t,\ell}+\nu\left(-\Delta\right)^{\gamma}\right)\widetilde{z}\left(s\right)\right\Vert _{\alpha}\;\mathrm{d}s\label{eq:modified_transpot}\\
 & \lesssim\ell^{\alpha}\tau_{q}^{-1}\int_{t_{j}}^{t}\left\Vert \widetilde{z}(s)\right\Vert _{\alpha}\;\mathrm{d}s+\epsilon_{q}^{2}\tau_{q}\delta_{q+1}\ell^{\alpha}\nonumber 
\end{align}
By Gronwall, we obtain (\ref{eq:zestimate}) for $N=0$. For $N\geq1,$
we observe that 
\begin{align*}
\left\Vert z_{j}-z_{\ell}\right\Vert _{N+\alpha}&\lesssim\left\Vert \nabla\left(z_{j}-z_{\ell}\right)\right\Vert _{N-1+\alpha}\\
&=\left\Vert \nabla\mathcal{B}\left(v_{j}-v_{\ell}\right)\right\Vert _{N-1+\alpha}\\
&\lesssim\left\Vert v_{j}-v_{\ell}\right\Vert _{N-1+\alpha},
\end{align*}
where we have implicitly used the facts that  $\nabla\mathcal{B}$
is Calderón-Zygmund, and that $\left\Vert f\right\Vert _{L^{\infty}}\lesssim\left\Vert \nabla f\right\Vert _{L^{\infty}}$
for any mean-zero $f\in C^{1}\left(\mathbb{T}^{d}\right)$ (Poincaré
inequality). Then by (\ref{eq:vestimate-1}), we obtain (\ref{eq:zestimate}).  From here, we note that (\ref{eq:Lztilde}) and (\ref{eq:zestimate}) imply (\ref{eq:ztransport_lambda_est}).

It remains to show (\ref{eq:ztransportestimate}).  For this, we argue as in (\ref{eq:new_in_glue}) and use (\ref{eq:diss_1}) to write
\begin{align*}
\|(-\Delta)^{\gamma}(z_{j}-z_{\ell})\|_{N+\alpha} & \lesssim\|z_{j}-z_{\ell}\|_{N+2\alpha+2\gamma}\\
&\lesssim\epsilon_{q}\tau_{q}\delta_{q+1}\ell^{-2\gamma-N}\\
&\lesssim\epsilon_{q}\delta_{q+1}\ell^{-N+\alpha},
\end{align*}
as desired.
\end{proof}
Combining (\ref{eq:cutoff_time_deri}), (\ref{eq:zestimate}) and
(\ref{eq:vestimate-1}), as well as the boundedness of the Calderón-Zygmund
operator $\mathcal{R}\delta$, we obtain
\begin{align}
\|\dd_{t}\chi_{j}^{b}\mathcal{R}(v_{j}-v_{j+1})\|_{N+\alpha}=\|\dd_{t}\chi_{j}^{b}\mathcal{R}\delta(z_{j}-z_{j+1})\|_{N+\alpha} & \lesssim_{N}\delta_{q+1}\ell^{-N+\alpha}\label{eq:1stRv},
\end{align}
and
\begin{align}
\|\chi_{j}^{b}(1-\chi_{j}^{b})(v_{j}-v_{j+1})\otimes(v_{j}-v_{j+1})\|_{N+\alpha} & \lesssim_{N}(\epsilon_{q}\tau_{q}\delta_{q+1}\ell^{-1+\alpha})^{2}\ell^{-N},\label{eq:2ndRV}
\end{align}
for $N\in\mathbb{N}_{0}$ and $t\in\left(t_{j+1},t_{j+1}+\epsilon_{q}\tau_{q}\right)$.

Before we proceed, we will need a usual singular-integral commutator
estimate from \cite{buckmasterOnsagerConjectureAdmissible2017}
to handle the Calderón-Zygmund operator $\mathcal{R}\curl$.
\begin{lem}[Proposition D.1 in \cite{buckmasterOnsagerConjectureAdmissible2017}]
\label{lem:singular_comm}Let $\alpha\in\left(0,1\right),N\in\mathbb{N}_{0}$,
$T$ be a Calderón-Zygmund operator and $b\in C^{N+1,\alpha}$ be
a divergence-free vector field on $\mathbb{T}^{d}$. Then for any
$f\in C^{N+\alpha}\left(\mathbb{T}^{d}\right)$, we have
\[
\left\Vert \left[T,b\cdot\nabla\right]f\right\Vert _{N+\alpha}\lesssim_{N,\alpha,T}\left\Vert b\right\Vert _{1+\alpha}\left\Vert f\right\Vert _{N+\alpha}+\left\Vert b\right\Vert _{N+1+\alpha}\left\Vert f\right\Vert _{\alpha}
\]
\end{lem}

We are now able to establish the relevant estimates for $\overline{R}_q$.

\begin{prop}
\label{prop:glued_stress_est}$\overline{R}_{q}$ in (\ref{eq:glued_stress})
admits the bounds 
\begin{align}
\|\overline{R}_{q}\|_{N+\alpha} & \lesssim_{N}\delta_{q+1}\ell^{-N+\alpha}\label{eq:gluedR}\\
\|(\dd_{t}+\overline{v}_{q}\cdot\grad)\overline{R}_{q}\|_{N+\alpha} & \lesssim_{N}(\epsilon_{q}\tau_{q})^{-1}\delta_{q+1}\ell^{-N+\alpha}\label{eq:gluedRtransport}
\end{align}
for $N\in\mathbb{N}_{0}$ and $t\in\left(t_{j+1},t_{j+1}+\epsilon_{q}\tau_{q}\right)$.
\end{prop}

\begin{proof}
We observe that (\ref{eq:1stRv}) and (\ref{eq:2ndRV}) imply
\[
\|\overline{R}_{q}\|_{N+\alpha}\lesssim_{N}\delta_{q+1}\ell^{-N+\alpha}(1+\epsilon_{q}\tau_{q}\delta_{q+1}\ell^{-1+\alpha})^{2},
\]
and then (\ref{eq:gluedR}) follows from (\ref{eq:time-length}).

On the other hand, we have 
\begin{align*}
\left\Vert \left(\partial_{t}+\nabla_{\overline{v}_{q}}\right)\overline{R}_{q}\right\Vert _{N+\alpha} & \leq\left\Vert D_{t,\ell}\overline{R}_{q}\right\Vert _{N+\alpha}+\left\Vert \nabla_{\overline{v}_{q}-v_{\ell}}\overline{R}_{q}\right\Vert _{N+\alpha}
\end{align*}
where 
\begin{align*}
D_{t,\ell}\overline{R}_{q} & =\left(\partial_{t}^{2}\chi_{j}^{b}\right)\mathcal{R}\delta\left(z_{j}-z_{j+1}\right)\\
 & +\left(\partial_{t}\chi_{j}^{b}\right)\mathcal{R}\delta D_{t,\ell}\left(z_{j}-z_{j+1}\right)+\left(\partial_{t}\chi_{j}^{b}\right)\left[v_{\ell}\cdot\nabla,\mathcal{R}\delta\right]\left(z_{j}-z_{j+1}\right)\\
 & +\partial_{t}\left(\left(\chi_{j}^{b}\right)^{2}-\chi_{j}^{b}\right)\left(v_{j}-v_{j+1}\right)\otimes\left(v_{j}-v_{j+1}\right)\\
 & +\left(\left(\chi_{j}^{b}\right)^{2}-\chi_{j}^{b}\right)\left(D_{t,\ell}\left(v_{j}-v_{j+1}\right)\otimes\left(v_{j}-v_{j+1}\right)\right.\\
&\hspace{1.8in}\left.+\left(v_{j}-v_{j+1}\right)\otimes D_{t,\ell}\left(v_{j}-v_{j+1}\right)\right)
\end{align*}

The term involving $\left[v_{\ell}\cdot\nabla,\mathcal{R}\delta\right]$
can be handled by \Lemref{singular_comm}. Then by (\ref{eq:cutoff_time_deri}),
(\ref{eq:gluedR}), \Propref[s]{vglue} and \ref{prop:zglue}, we
conclude 
\begin{align*}
\|(\dd_{t}+\overline{v}_{q}\cdot\grad)\overline{R}_{q}\|_{N+\alpha} & \lesssim_{N}\left(\epsilon_{q}\tau_{q}\right)^{-1}\delta_{q+1}\ell^{-N+\alpha}\\
&\hspace{0.4in}+\tau_{q}^{-1}\delta_{q+1}\ell^{-N+\alpha}\\
&\hspace{0.4in}+\epsilon_{q}\tau_{q}\delta_{q+1}^{2}\ell^{-2-N+2\alpha}
\end{align*}
which then yields (\ref{eq:gluedRtransport}) because of (\ref{eq:time-length}).
\end{proof}

\subsection{Good-bad interface}

Next we consider any pair of indices $i$ and $j$ such that $\chi_{i}^{g}\chi_{j}^{b}\not\equiv0$.
By construction, we observe that $\supp(\chi_{i}^{g}\chi_{j}^{b})$
lies in an interval of length $\sim\epsilon_{q}\tau_{q}$, where $R_{q}$
is 0.

Without loss of generality (i.e., depending on whether $\chi_{i}^{g}$
or $\chi_{j}^{b}$ comes first in time), in this interval $\overline{v}_{q}$
satisfies 
\[
\dd_{t}\overline{v}_{q}+\nu(-\Delta)^{\gamma}\overline{v}_{q}+\divop\overline{v}_{q}\otimes\overline{v}_{q}+\grad\overline{p}_{q}=\divop\overline{R}_{q}
\]
where 
\begin{equation}
\overline{R}_{q}=\dd_{t}\chi_{i}^{g}\mathcal{R}(v_{q}-v_{j})-\chi_{i}^{g}(1-\chi_{i}^{g})(v_{q}-v_{j})\otimes(v_{q}-v_{j})\label{eq:glued_stress-1}
\end{equation}
which is a perfect analogue of (\ref{eq:glued_stress}).

As before, we decompose 
\[
v_{q}-v_{j}=(v_{q}-v_{\ell})-(v_{j}-v_{\ell})
\]
The estimates for $v_{j}-v_{\ell}$ are exactly as above.  Turning to $v_{q}-v_{\ell}$, the relevant estimates are given
by the following result.
\begin{prop}
\label{prop:vglue-1}For $N\in\mathbb{N}_{0}$ and $t\in\mathcal{G}_{q}+B\left(0,\tau_{q}+\epsilon_{q}\tau_{q}\right)$:
\begin{align}
\|v_{q}-v_{\ell}\|_{N+\alpha} & \lesssim_{N}\epsilon_{q}\tau_{q}\delta_{q+1}\ell^{-N-1+\alpha}\label{eq:vestimate-1-1}\\
\|(\dd_{t}+v_{\ell}\cdot\grad+\nu \left(-\Delta\right)^{\gamma})\left(v_{q}-v_{\ell}\right)\|_{N+\alpha} & \lesssim_{N}\epsilon_{q}\delta_{q+1}\ell^{-N-1+\alpha}\label{eq:vlestimate-1}\\
\|D_{t,\ell}\left(v_{q}-v_{\ell}\right)\|_{N+\alpha} & \lesssim_{N}\epsilon_{q}\delta_{q+1}\ell^{-N-1+\alpha}\label{eq:vtransportestimate-1-1}
\end{align}
\end{prop}

\begin{proof}
By standard mollification estimates (cf. \cite[Lemma 2.1]{contiHPrincipleRigidityAlpha2009}), we have
\begin{align*}
\|v_{q}-v_{\ell}\|_{N+\alpha} & \lesssim_{N}\ell_{q}\|v_{q}\|_{N+1+\alpha}\\
 & \lesssim_{N}\delta_{q-1}^{\frac{1}{2}}\lambda_{q-1}\ell_{q}\ell_{q-1}^{-N-\alpha}\ll\delta_{q-1}^{\frac{1}{2}}\lambda_{q-1}\ell_{q}^{1-N-\alpha}\\
 & \lesssim\epsilon_{q}\tau_{q}\delta_{q+1}\ell_{q}^{-N-1+\alpha}
\end{align*}
where we used (\ref{eq:goodbounds}) to pass to the second line, and
(\ref{eq:good_bad1}) to pass to the last line. Thus (\ref{eq:vestimate-1-1})
is proven.

Then as $R_{q}=0$ on this temporal region, we have an analogue of
(\ref{eq:vl_vj_subtract}) and (\ref{eq:vl_vj_subtract2}), namely
\begin{align}
\left(\partial_{t}+v_{\ell}\cdot\nabla+\nu\left(-\Delta\right)^{\gamma}\right)\left(v_{\ell}-v_{q}\right) & =-\left(v_{\ell}-v_{q}\right)\cdot\nabla v_{q}-\nabla\left(p_{\ell}-p_{q}\right)+\Div R_{\ell}\label{eq:vl_vj_subtract-1}
\end{align}
and
\begin{align}
\nabla\left(p_{\ell}-p_{q}\right) & =\mathcal{P}_{1}\left(-\left(v_{\ell}-v_{q}\right)\cdot\nabla v_{\ell}-\left(v_{\ell}-v_{q}\right)\cdot\nabla v_{q}+\Div R_{\ell}\right)
\end{align}
Thus we can estimate $\left\Vert \nabla\left(p_{\ell}-p_{q}\right)\right\Vert _{N+\alpha}$
and then $$\|(\dd_{t}+v_{\ell}\cdot\grad+\nu\left(-\Delta\right)^{\gamma})\left(v_{q}-v_{\ell}\right)\|_{N+\alpha}$$
to obtain (\ref{eq:vlestimate-1}). We then argue as in (\ref{eq:new_in_glue})
(replacing $v_{j}$ by $v_{q}$) to obtain (\ref{eq:vtransportestimate-1-1}).
\end{proof}
Note that we have fully proven (\ref{eq:glue1}).

To proceed, we define the potentials $z_{q}:=\mathcal{B}v_{q},z_{\ell}:=\mathcal{B}v_{\ell}$.
By observing that \Propref{vglue-1} plays the exact same role as
\Propref{vglue}, and by arguing exactly as in \Propref{zglue} (replacing
$v_{j}$ with $v_{q}$, and $z_{j}$ with $z_{q}$) we obtain
\begin{align}
\|z_{q}-z_{\ell}\|_{N+\alpha} & \lesssim_{N}\epsilon_{q}\tau_{q}\delta_{q+1}\ell^{-N+\alpha}\label{eq:zestimate-1}\\
\|(\dd_{t}+v_{\ell}\cdot\grad+\nu\left(-\Delta\right)^{\gamma})\left(z_{q}-z_{\ell}\right)\|_{N+\alpha} & \lesssim_{N}\epsilon_{q}\delta_{q+1}\ell^{-N+\alpha}\label{eq:ztransport_lambda_est-1}\\
\|D_{t,\ell}\left(z_{q}-z_{\ell}\right)\|_{N+\alpha} & \lesssim_{N}\epsilon_{q}\delta_{q+1}\ell^{-N+\alpha}\label{eq:ztransportestimate-1}
\end{align}
 for any $N\in\mathbb{N}_{0}$ and $t\in\mathcal{G}_{q}+B\left(0,\tau_{q}+\epsilon_{q}\tau_{q}\right)$.

Then, as with (\ref{eq:1stRv}) and (\ref{eq:2ndRV}), we have 
\begin{align}
\|\dd_{t}\chi_{i}^{g}\mathcal{R}(v_{q}-v_{j})\|_{N+\alpha} & \lesssim_{N}\delta_{q+1}\ell^{-N+\alpha}\label{eq:1stRv-1}\\
\|\chi_{j}^{g}(1-\chi_{j}^{g})(v_{q}-v_{j})\otimes(v_{q}-v_{j})\|_{N+\alpha} & \lesssim_{N}(\epsilon_{q}\tau_{q}\delta_{q+1}\ell^{-1+\alpha})^{2}\ell^{-N}\label{eq:2ndRV-1}
\end{align}
 for any $N\in\mathbb{N}_{0}$ and $t\in\supp(\chi_{i}^{g}\chi_{j}^{b})$.

We also have the analogue of \Propref{glued_stress_est}. By making
the obvious replacements ($v_{j}-v_{j+1}$ with $v_{q}-v_{j}$, $z_{j}-z_{j+1}$
with $z_{q}-z_{j}$, and $\chi_{j}^{b}$ with $\chi_{i}^{g}$), we
have 
\begin{align}
\|\overline{R}_{q}\|_{N+\alpha} & \lesssim_{N}\delta_{q+1}\ell^{-N+\alpha}\label{eq:gluedR-1}\\
\|(\dd_{t}+\overline{v}_{q}\cdot\grad)\overline{R}_{q}\|_{N+\alpha} & \lesssim_{N}(\epsilon_{q}\tau_{q})^{-1}\delta_{q+1}\ell^{-N+\alpha}\label{eq:gluedRtransport-1}
\end{align}
 for any $N\in\mathbb{N}_{0}$ and $t\in\supp(\chi_{i}^{g}\chi_{j}^{b})$.

\section{\label{sec:Convex-integration-and}Perturbation estimates}

In this section, we prove \Propref{convex_int}, the perturbation result which was used in the proof of Proposition \ref{prop:iterative} (in Section \ref{sec:Proof-of-iteration}).  We begin by recalling the definition of the Mikado flows from \cite[Lemma 5.1]{buckmasterOnsagerConjectureAdmissible2017},
which is valid for any dimension $d\geq3$ (see also \cite[Section 4.1]{cheskidovSharpNonuniquenessNavierStokes2020}).  

For any compact subset $\mathcal{N}\subset\subset\mathcal{S}_{+}^{d\times d}$, there is a smooth vector field $W:\mathcal{N}\times\mathbb{T}^{d}\to\mathbb{R}^{d}$
such that 
\begin{align}
\Div_{\xi}W(R,\xi)\otimes W(R,\xi) & =0\label{eq:divf1}\\
\Div_{\xi}W(R,\xi) & =0\label{eq:divf2}\\
\dashint_{\mathbb{T}^{d}}W\left(R,\xi\right)\;\mathrm{d}\xi & =0\\
\dashint_{\mathbb{T}^{d}}W(R,\xi)\otimes W(R,\xi)\;\mathrm{d}\xi & =R
\end{align}
Unless otherwise noted, we set $\mathcal{N}=\overline{B_{1/2}(\textrm{Id)}}$.

By Fourier decomposition we have 
\begin{align*}
W\left(R,\xi\right) & =\sum_{k\in\mathbb{Z}^{d}\backslash\{0\}}a_{k}\left(R\right)e^{i2\pi\left\langle k,\xi\right\rangle },
\end{align*}
and
\begin{align*}
W(R,\xi)\otimes W(R,\xi) & =R+\sum_{k\in\mathbb{Z}^{d}\backslash\{0\}}C_{k}\left(R\right)e^{i2\pi\left\langle k,\xi\right\rangle },
\end{align*}
where $a_{k}\left(R\right)$ and $C_{k}\left(R\right)$ are smooth
in $R$, and with derivatives rapidly decaying in $k$. Furthermore,
(\ref{eq:divf1}) and (\ref{eq:divf2}) imply
\begin{align}
k\cdot a_{k}\left(R\right) & =0\label{eq:k_ak}
\end{align}
and
\begin{align}
k^{\flat}\lrcorner C_{k}\left(R\right) & =0.\label{eq:kflat_c}
\end{align}

Now we recall the identity 
\[
v\lrcorner\left(\alpha\wedge\beta\right)=\left(v\lrcorner\alpha\right)\wedge\beta-\alpha\wedge\left(v\lrcorner\beta\right)
\]
for any vector field $v$, $1$-form $\alpha$, and differential form
$\beta$. This implies 
\begin{align}
\Div{}_{\xi}\left(\frac{k\wedge a_{k}}{i2\pi\left|k\right|^{2}}e^{i2\pi\left\langle k,\xi\right\rangle }\right) & =-\sharp\delta_{\xi}\left(\frac{k^{\flat}\wedge a_{k}^{\flat}}{i2\pi\left|k\right|^{2}}e^{i2\pi\left\langle k,\xi\right\rangle }\right)\label{eq:freeze_codiff}\\
 & =\sharp\left(e^{i2\pi\left\langle k,\xi\right\rangle }i2\pi k\right)\lrcorner\left(\frac{k^{\flat}\wedge a_{k}^{\flat}}{i2\pi\left|k\right|^{2}}\right)=a_{k}e^{i2\pi\left\langle k,\xi\right\rangle }\label{eq:freeze_2}
\end{align}
where $k\wedge a_{k}$ is an alternating $(2,0)$-tensor dual to $k^{\flat}\wedge a_{k}^{\flat}$.
Note that we implicitly used (\ref{eq:k_ak}).
%%%%%

To handle the transport error later and generalize the ``vector calculus''
to higher dimensions, we also introduce a local-time version of Lagrangian
coordinates.\footnote{The formalism is discussed in Tao's lecture notes, which can be found
at \url{https://terrytao.wordpress.com/2019/01/08/255b-notes-2-onsagers-conjecture/}}
\begin{defn}[Lagrangian coordinates]
We define the backwards transport flow $\Phi_{i}$ as the solution
to 
\begin{align*}
\left(\partial_{t}+\overline{v}_{q}\cdot\nabla\right)\Phi_{i} & =0\\
\Phi_{i}\left(t_{i},\cdot\right) & =\mathrm{Id}_{\mathbb{T}^{d}}
\end{align*}
Then as in \cite[Proposition 3.1]{buckmasterOnsagerConjectureAdmissible2017},
for any $N\geq2$ and $\left|t-t_{i}\right|\lesssim\tau_{q}$:
\begin{align}
\left\Vert \nabla\Phi_{i}\left(t\right)-\mathrm{Id}\right\Vert _{0} & \lesssim\left|t-t_{i}\right|\left\Vert \nabla\overline{v}_{q}\right\Vert _{0}\lesssim\tau_{q}\delta_{q}^{\frac{1}{2}}\lambda_{q}=\lambda_{q}^{-3\alpha}\ll1\label{eq:phi_id}\\
\left\Vert \nabla^{N}\Phi_{i}\left(t\right)\right\Vert _{0} & \lesssim\left|t-t_{i}\right|\left\Vert \nabla^{N}\overline{v}_{q}\right\Vert _{0}\lesssim\lambda_{q}^{-3\alpha}\ell^{-N+1}\quad\label{eq:phi_id_2}
\end{align}
We also define the forward characteristic flow $X_{i}$ as the the
flow generated by $\overline{v}_{q}$: 
\begin{align*}
\partial_{\underline{t}}X_{i}\left(\underline{t},\underline{x}\right) & =\overline{v}_{q}\left(\underline{t},X_{i}\left(\underline{t},\underline{x}\right)\right)\\
X_{i}\left(t_{i},\cdot\right) & =\mathrm{Id}_{\mathbb{T}^{d}}
\end{align*}
Then $\partial_{\underline{t}}\left(\Phi_{i}\left(\underline{t},X_{i}\left(\underline{t},\underline{x}\right)\right)\right)=0$.
By defining their spacetime versions
\begin{align*}
\mathbf{\Phi}_{i}\left(t,x\right) & :=\left(t,\Phi_{i}\left(t,x\right)\right)\\
\mathbf{X}_{i}\left(\underline{t},\underline{x}\right) & :=\left(\underline{t},X_{i}\left(\underline{t},\underline{x}\right)\right)
\end{align*}
we can conclude $\mathbf{X}_{i}=\left(\mathbf{\Phi}_{i}\right)^{-1}$,
and that $\mathbf{X}_{i}$ maps from the Lagrangian spacetime $\left(\underline{t},\underline{x}\right)$
to the Eulerian spacetime $\left(t,x\right)$. 

Let $\vol$ be the standard volume form of the torus. Then $X_{i}\left(\underline{t}\right)^{*}\vol=\vol$
(volume-preserving)\footnote{$\partial_{\underline{t}}\left(X_{i}\left(\underline{t}\right)^{*}\vol\right)=X_{i}\left(\underline{t}\right)^{*}\left(\mathcal{L}_{\overline{v}_{q}\left(\underline{t}\right)}\vol\right)=X_{i}\left(\underline{t}\right)^{*}\left(\Div\overline{v}_{q}\left(\underline{t}\right)\vol\right)=0.$
See also (\ref{eq:pullback_lie}).} and 
\begin{equation}
X_{i}\left(\underline{t}\right)^{*}\left(\Div u\right)=\Div\left(X_{i}\left(\underline{t}\right)^{*}u\right)\label{eq:Div_commute}
\end{equation}
 for any vector field $u$.\footnote{$X_{i}\left(\underline{t}\right)^{*}\left(\Div u\right)\vol=X_{i}\left(\underline{t}\right)^{*}\left(\mathcal{L}_{u}\vol\right)=\mathcal{L}_{X_{i}\left(\underline{t}\right)^{*}u}X_{i}\left(\underline{t}\right)^{*}\vol=\Div\left(X_{i}\left(\underline{t}\right)^{*}u\right)\vol$}
\end{defn}

\subsection{Constructing the perturbation}
We now specify the key terms used to define our perturbation.  Set 
\[
\underline{R_{i}}:=\mathbf{X}_{i}^{*}\left(\textrm{Id}-\frac{\overline{R}_{q}}{\delta_{q+1}}\right)
\]
where we treat $\overline{R}_{q}$ as a $(2,0)$-tensor.  Indeed, we can write this more explicitly as
\begin{equation}
\underline{R_{i}}\circ\mathbf{\Phi}_{i}=\nabla\Phi_{i}\left(\textrm{Id}-\frac{\overline{R}_{q}}{\delta_{q+1}}\right)\nabla\Phi_{i}^{T}\label{eq:R_phiii}
\end{equation}
Note that, for $\left|t-t_{i}\right|\lesssim\tau_{q}$, we have 
\[
\underline{R_{i}}\circ\mathbf{\Phi}_{i}\in B_{1/2}(\textrm{Id)},
\]
because $\grad\Phi_{i}$ is close to $\textrm{Id}$ and $\left\Vert \frac{\overline{R}_{q}}{\delta_{q+1}}\right\Vert _{0}\lesssim\ell^{\alpha}$ by (\ref{eq:glue3}).

For each $i$ let $\rho_{i}$ be a smooth cutoff such that $\mathbf{1}_{\left[t_{i},t_{i}+\epsilon_{q}\tau_{q}\right]}\leq\rho_{i}\leq\mathbf{1}_{\left[t_{i}-\epsilon_{q}\tau_{q},t_{i}+2\epsilon_{q}\tau_{q}\right]}$
and satisfying the estimate
\[
\left\Vert \partial_{t}^{N}\rho_{i}\right\Vert _{0}\lesssim\left(\epsilon_{q}\tau_{q}\right)^{-N}\;\forall N\in\mathbb{N}_{0}
\]
We now define the perturbation 
\begin{align*}
w^{(o)} & :=\sum_{i}\delta_{q+1}^{1/2}\rho_{i}(t)\grad\Phi_{i}^{-1}W(\underline{R_{i}}\circ\mathbf{\Phi}_{i},\lambda_{q+1}\Phi_{i}).
\end{align*}

For $t\in\left[t_{i}-\epsilon_{q}\tau_{q},t_{i}+2\epsilon_{q}\tau_{q}\right]$,
in local-time Lagrangian coordinates with $$\underline{w^{(o)}}:=\mathbf{X}_{i}^{*}w^{(o)},$$
we have 
\begin{align*}
\underline{w^{(o)}} & =\delta_{q+1}^{1/2}\rho_{i}(\underline{t})W(\underline{R_{i}},\lambda_{q+1}\underline{x})\\
 & =\sum_{k\neq0}\underbrace{\delta_{q+1}^{1/2}\rho_{i}(\underline{t})a_{k}(\underline{R_{i}})}_{:=\underline{b_{i,k}}}e^{i2\pi\left\langle \lambda_{q+1}k,\underline{x}\right\rangle }=\sum_{k\neq0}\underline{b_{i,k}}e^{i2\pi\left\langle \lambda_{q+1}k,\underline{x}\right\rangle }
\end{align*}
and therefore, by defining $b_{i,k}:=\mathbf{\Phi}_{i}^{*}\underline{b_{i,k}}$
(zero-extended outside $\supp\rho_{i}$), we have
\[
w^{\left(o\right)}=\sum_{i}\sum_{k\neq0}b_{i,k}e^{i2\pi\left\langle \lambda_{q+1}k,\Phi_{i}\right\rangle }
\]

Now, for $t\in\left[t_{i}-\epsilon_{q}\tau_{q},t_{i}+2\epsilon_{q}\tau_{q}\right]$,
in local-time Lagrangian coordinates, we define the incompressibility
corrector 
\[
\underline{w^{\left(c\right)}}:=\sum_{k\neq0}\underbrace{\delta_{q+1}^{1/2}\rho_{i}(\underline{t})\Div_{\underline{x}}\left(\frac{k\wedge a_{k}\left(\underline{R_{i}}\right)}{i2\pi\lambda_{q+1}\left|k\right|^{2}}\right)}_{:=\underline{c_{i,k}}}e^{i2\pi\left\langle \lambda_{q+1}k,\underline{x}\right\rangle }=\sum_{k\neq0}\underline{c_{i,k}}e^{i2\pi\left\langle \lambda_{q+1}k,\underline{x}\right\rangle }.
\]

Because of (\ref{eq:freeze_2}) and the identity $$\Div\left(fv\right)=\nabla f\lrcorner v^{\flat}+f\Div v,$$ which holds
for any smooth function $f$ and vector field $v$, we have 
\begin{align*}
\underline{w^{(o)}}+\underline{w^{\left(c\right)}} & =\delta_{q+1}^{1/2}\rho_{i}(\underline{t})\sum_{k\neq0}e^{i2\pi\left\langle \lambda_{q+1}k,\underline{x}\right\rangle }\left(a_{k}(\underline{R_{i}})+\Div_{\underline{x}}\left(\frac{k\wedge a_{k}\left(\underline{R_{i}}\right)}{i2\pi\lambda_{q+1}\left|k\right|^{2}}\right)\right)\\
 & =\delta_{q+1}^{1/2}\rho_{i}(\underline{t})\sum_{k\neq0}\Div_{\underline{x}}\left(\frac{k\wedge a_{k}\left(\underline{R_{i}}\right)}{i2\pi\lambda_{q+1}\left|k\right|^{2}}e^{i2\pi\left\langle \lambda_{q+1}k,\underline{x}\right\rangle }\right)
\end{align*}
which is divergence-free, since $\Div\Div T=0$ for any alternating
$(2,0)$-tensor on the flat torus.

In Eulerian coordinates, we define $c_{i,k}:=\mathbf{\Phi}_{i}^{*}\underline{c_{i,k}}$
(zero-extended outside $\supp\rho_{i}$), as well as
\begin{align*}
w^{(c)} & :=\sum_{i}\sum_{k\neq0}c_{i,k}e^{i2\pi\left\langle \lambda_{q+1}k,\Phi_{i}\right\rangle }
\end{align*}
to obtain $\underline{w^{(c)}}=\mathbf{X}_{i}^{*}w^{(c)}$ for $t\in\left[t_{i}-\epsilon_{q}\tau_{q},t_{i}+2\epsilon_{q}\tau_{q}\right]$.

Because of (\ref{eq:Div_commute}), the full perturbation
\[
w_{q+1}:=w^{(o)}+w^{(c)}
\]
is divergence-free.

With these ingredients in place, we now define 
\[
v_{q+1}:=\overline{v}_{q}+w_{q+1}
\]
 and observe that
\begin{align*}
 & \partial_{t}v_{q+1}+\Div\left(v_{q+1}\otimes v_{q+1}\right)+\nu\left(-\Delta\right)^{\gamma}v_{q+1}\\
=\; & \left(\partial_{t}\overline{v}_{q}+\Div\left(\overline{v}_{q}\otimes\overline{v}_{q}\right)+\nu\left(-\Delta\right)^{\gamma}\overline{v}_{q}\right)+\Div\left(w_{q+1}\otimes w_{q+1}\right)\\
 & \quad+\partial_{t}w_{q+1}+\Div\left(\overline{v}_{q}\otimes w_{q+1}\right)+\Div\left(w_{q+1}\otimes\overline{v}_{q}\right)+\nu\left(-\Delta\right)^{\gamma}w_{q+1}\\
=\; & -\nabla\overline{p}_{q}+\Div\left(\overline{R}_{q}+w_{q+1}\otimes w_{q+1}\right)\\
 & \quad+D_{t,q}w_{q+1}+w_{q+1}\cdot\nabla\overline{v}_{q}+\nu\left(-\Delta\right)^{\gamma}w_{q+1}
\end{align*}
We can then define the final stress as 
\begin{align*}
R_{q+1} & :=R_{\mathrm{osc}}+R_{\mathrm{trans}}+R_{\mathrm{Nash}}+R_{\mathrm{dis}}\\
R_{\mathrm{osc}} & :=\mathcal{R}\Div\left(\overline{R}_{q}+w_{q+1}\otimes w_{q+1}\right)\\
R_{\mathrm{trans}} & :=\mathcal{R}D_{t,q}w_{q+1}\\
R_{\mathrm{Nash}} & :=\mathcal{R}\left(w_{q+1}\cdot\nabla\overline{v}_{q}\right)\\
R_{\mathrm{dis}} & :=\nu\mathcal{R}\left(-\Delta\right)^{\gamma}w_{q+1}
\end{align*}

\subsection{Perturbation estimates}

To establish \Propref{convex_int}, we now have to estimate the perturbation constructed in the previous subsection.  The desired bounds are established in the following
series of results.  

\begin{prop}
\label{prop:pertub_est}
Suppose $t\in\left[t_{i}-\epsilon_{q}\tau_{q},t_{i}+2\epsilon_{q}\tau_{q}\right]$ and $N\in\mathbb{N}_{0}$.  Then we have the following estimates,
\begin{align}
\left\Vert \nabla\Phi_{i}\right\Vert _{N}+\left\Vert \nabla\Phi_{i}^{-1}\right\Vert _{N} & \lesssim_{N}\ell^{-N}\label{eq:per_est1}\\
\left\Vert \underline{R_{i}}\circ\mathbf{\Phi}_{i}\right\Vert _{N} & \lesssim_{N}\ell^{-N}\label{eq:per_est2}\\
\left\Vert b_{i,k}\right\Vert _{N} & \lesssim_{N}\delta_{q+1}^{\frac{1}{2}}\ell^{-N}\left|k\right|^{-2d}\label{eq:per_est3}\\
\left\Vert c_{i,k}\right\Vert _{N} & \lesssim_{N}\delta_{q+1}^{\frac{1}{2}}\lambda_{q+1}^{-1}\ell^{-N-1}\left|k\right|^{-2d}\label{eq:per_est4}
\end{align}
along with their material derivative analogues,
\begin{align}
\left\Vert D_{t,q}\left(\nabla\Phi_{i}\right)\right\Vert _{N} & \lesssim_{N}\delta_{q}^{\frac{1}{2}}\lambda_{q}\ell^{-N}\label{eq:per_est5}\\
\left\Vert D_{t,q}\left(\underline{R_{i}}\circ\mathbf{\Phi}_{i}\right)\right\Vert _{N} & \lesssim_{N}\left(\epsilon_{q}\tau_{q}\right)^{-1}\ell^{-N+\alpha}\label{eq:per_est6}\\
\left\Vert D_{t,q}b_{i,k}\right\Vert _{N} & \lesssim_{N}\left(\epsilon_{q}\tau_{q}\right)^{-1}\delta_{q+1}^{1/2}\ell^{-N}\left|k\right|^{-2d}\label{eq:per_est7}\\
\left\Vert D_{t,q}c_{i,k}\right\Vert _{N} & \lesssim_{N}\left(\epsilon_{q}\tau_{q}\right)^{-1}\delta_{q+1}^{1/2}\lambda_{q+1}^{-1}\ell^{-N-1}\left|k\right|^{-2d}\label{eq:per_est8}
\end{align}
\end{prop}

\begin{proof}
We first observe that (\ref{eq:phi_id}) and (\ref{eq:phi_id_2})
imply $\left\Vert \nabla\Phi_{i}\right\Vert _{N}\lesssim\ell^{-N}$.
Then the fact that $\nabla\Phi_{i}$ is close to $\mathrm{Id}$, and
the elementary identity $d\left(A^{-1}\right)=-A^{-1}\left(dA\right)A^{-1}$
(for any invertible matrix $A$) imply (\ref{eq:per_est1}).

Next, we observe that (\ref{eq:per_est1}) and (\ref{eq:glue3}) imply
(\ref{eq:per_est2}), via the bounds
\[
\left\Vert \underline{R_{i}}\circ\mathbf{\Phi}_{i}\right\Vert _{N}\lesssim_{N}\left\Vert \nabla\Phi_{i}\right\Vert _{0}^{2}\left\Vert \textrm{Id}-\frac{\overline{R}_{q}}{\delta_{q+1}}\right\Vert _{N}+\left\Vert \nabla\Phi_{i}\right\Vert _{N}\left\Vert \nabla\Phi_{i}\right\Vert _{0}\left\Vert \textrm{Id}-\frac{\overline{R}_{q}}{\delta_{q+1}}\right\Vert _{0}\lesssim\ell^{-N}
\]
Then, because of (\ref{eq:per_est2}), and the fact that the derivatives
of $a_{k}$ rapidly decay in $k$, we obtain
\begin{align*}
\left\Vert b_{i,k}\right\Vert _{N} & =\left\Vert \delta_{q+1}^{1/2}\rho_{i}(t)\nabla\Phi_{i}^{-1}a_{k}\left(\underline{R_{i}}\circ\mathbf{\Phi}_{i}\right)\right\Vert _{N}\\
 & \lesssim\delta_{q+1}^{1/2}\left(\left\Vert \nabla\Phi_{i}^{-1}\right\Vert _{N}\left\Vert a_{k}\left(\underline{R_{i}}\circ\mathbf{\Phi}_{i}\right)\right\Vert _{0}+\left\Vert \nabla\Phi_{i}^{-1}\right\Vert _{0}\left\Vert a_{k}\left(\underline{R_{i}}\circ\mathbf{\Phi}_{i}\right)\right\Vert _{N}\right)\\
 & \lesssim\delta_{q+1}^{1/2}\ell^{-N}\left|k\right|^{-2d},
\end{align*}
which establishes (\ref{eq:per_est3}).

Similarly we obtain (\ref{eq:per_est4}) by writing,
\begin{align*}
\left\Vert c_{i,k}\right\Vert _{N} & =\left\Vert \delta_{q+1}^{1/2}\rho_{i}(t)\nabla\Phi_{i}^{-1}\Div_{\underline{x}}\left(\frac{k\wedge a_{k}\left(\underline{R_{i}}\right)}{i2\pi\lambda_{q+1}\left|k\right|^{2}}\right)\circ\mathbf{\Phi}_{i}\right\Vert _{N}\\
 & \lesssim\delta_{q+1}^{1/2}\left|k\right|^{-1}\lambda_{q+1}^{-1}\left(\left\Vert \nabla\Phi_{i}^{-1}\right\Vert _{N}\left\Vert \nabla\left(a_{k}\left(\underline{R_{i}}\right)\right)\circ\mathbf{\Phi}_{i}\right\Vert _{0}\right.\\
&\hspace{1.6in}\left.+\left\Vert \nabla\Phi_{i}^{-1}\right\Vert _{0}\left\Vert \nabla\left(a_{k}\left(\underline{R_{i}}\right)\right)\circ\mathbf{\Phi}_{i}\right\Vert _{N}\right)\\
 & \lesssim\delta_{q+1}^{1/2}\left|k\right|^{-2d}\lambda_{q+1}^{-1}\ell^{-N-1},
\end{align*}
where we have implicitly used the chain rule 
\begin{equation}
\nabla\left(a_{k}\left(\underline{R_{i}}\right)\right)\circ\mathbf{\Phi}_{i}=\nabla\left(a_{k}\left(\underline{R_{i}}\circ\mathbf{\Phi}_{i}\right)\right)\left(\nabla\Phi_{i}\right)^{-1}\label{eq:chain_r}
\end{equation}
in passing to the last line.

We now turn to (\ref{eq:per_est5}), writing
\begin{align*}
\left\Vert D_{t,q}\nabla\Phi_{i}\right\Vert _{N} & =\left\Vert \nabla_{\overline{v}_{q}}\left(\nabla\Phi_{i}\right)+\nabla\partial_{t}\Phi_{i}\right\Vert _{N}\\
&=\left\Vert \left[\nabla_{\overline{v}_{q}},\nabla\right]\Phi_{i}\right\Vert _{N}\\
&\lesssim\left\Vert \nabla\overline{v}_{q}\right\Vert _{N}\left\Vert \nabla\Phi_{i}\right\Vert _{0}+\left\Vert \nabla\overline{v}_{q}\right\Vert _{0}\left\Vert \nabla\Phi_{i}\right\Vert _{N}\\
&\lesssim\delta_{q}^{\frac{1}{2}}\lambda_{q}\ell^{-N}.
\end{align*}

Next, we note that (\ref{eq:per_est5}), (\ref{eq:R_phiii}), (\ref{eq:glue3})
and (\ref{eq:glue4}) imply
\begin{align*}
&\left\Vert D_{t,q}\left(\underline{R_{i}}\circ\mathbf{\Phi}_{i}\right)\right\Vert _{N} \\
& \hspace{0.2in}\lesssim\left\Vert D_{t,q}\left(\nabla\Phi_{i}\right)\left(\textrm{Id}-\frac{\overline{R}_{q}}{\delta_{q+1}}\right)\nabla\Phi_{i}^{T}+\nabla\Phi_{i}\left(\textrm{Id}-\frac{\overline{R}_{q}}{\delta_{q+1}}\right)D_{t,q}\nabla\Phi_{i}^{T}\right\Vert _{N}\\
 & \hspace{0.4in}+\delta_{q+1}^{-1}\left\Vert \nabla\Phi_{i}\left(D_{t,q}\overline{R}_{q}\right)\nabla\Phi_{i}^{T}\right\Vert _{N}\\
 & \hspace{0.2in}\lesssim\delta_{q}^{\frac{1}{2}}\lambda_{q}\ell^{-N}+\left(\epsilon_{q}\tau_{q}\right)^{-1}\ell^{-N+\alpha}\\
&\hspace{0.2in}\lesssim\left(\epsilon_{q}\tau_{q}\right)^{-1}\ell^{-N+\alpha},
\end{align*}
establishing (\ref{eq:per_est6}), where in passing to the last inequality, we have implicitly used $\delta_{q}^{\frac{1}{2}}\lambda_{q}\ll\epsilon_{q}^{-1}\tau_{q}^{-1}\ell^{\alpha}$,
which comes from $\epsilon_{q}\ll1$ (after $\alpha$ is neglected).

Turning to (\ref{eq:per_est7}), we recall the identity $\partial_{\underline{t}}\left(w\circ\mathbf{X}_{i}\right)=\left(D_{t,q}w\right)\circ\mathbf{X}_{i}$
(for any tensor $w$). We then use (\ref{eq:per_est6}), (\ref{eq:per_est2}), and (\ref{eq:per_est1}) to write
\begin{align*}
\left\Vert D_{t,q}b_{i,k}\right\Vert _{N}&=\left\Vert \partial_{\underline{t}}\left(\mathbf{\Phi}_{i}^{*}\underline{b_{i,k}}\circ\mathbf{X}_{i}\right)\circ\mathbf{\Phi}_{i}\right\Vert _{N}\\
&\hspace{0.2in}=\left\Vert \partial_{\underline{t}}\left(\left(\nabla X_{i}\right)\underline{b_{i,k}}\right)\circ\mathbf{\Phi}_{i}\right\Vert _{N}\\
&\hspace{0.2in}=\delta_{q+1}^{1/2}\left\Vert \partial_{\underline{t}}\left(\left(\nabla X_{i}\right)\rho_{i}\left(\underline{t}\right)a_{k}\left(\underline{R_{i}}\right)\right)\circ\mathbf{\Phi}_{i}\right\Vert _{N}.
\end{align*}
The right-hand side of the above is now bounded by
\begin{align*}
&\lesssim\delta_{q+1}^{1/2}\left(\epsilon_{q}\tau_{q}\right)^{-1}\left\Vert \left(\left(\nabla X_{i}\right)a_{k}\left(\underline{R_{i}}\right)\right)\circ\mathbf{\Phi}_{i}\right\Vert _{N}\\
&\hspace{0.2in}+\delta_{q+1}^{1/2}\left\Vert \partial_{\underline{t}}\left(\left(\nabla X_{i}\right)a_{k}\left(\underline{R_{i}}\right)\right)\circ\mathbf{\Phi}_{i}\right\Vert _{N}\\
&\lesssim\delta_{q+1}^{1/2}\left(\epsilon_{q}\tau_{q}\right)^{-1}\left\Vert \left(\nabla\Phi_{i}\right)^{-1}a_{k}\left(\underline{R_{i}}\circ\mathbf{\Phi}_{i}\right)\right\Vert _{N}\\
&\hspace{0.2in}+\delta_{q+1}^{1/2}\left\Vert \left(\nabla\left(\overline{v}_{q}\circ\mathbf{X}_{i}\right)a_{k}\left(\underline{R_{i}}\right)\right)\circ\mathbf{\Phi}_{i}\right\Vert _{N}\\
&\hspace{0.2in}+\delta_{q+1}^{1/2}\left\Vert \left(\nabla X_{i}\right)\left(\nabla a_{k}\left(\underline{R_{i}}\right)\partial_{\underline{t}}\left(\underline{R_{i}}\right)\right)\circ\mathbf{\Phi}_{i}\right\Vert _{N}\\
&\lesssim\left(\epsilon_{q}\tau_{q}\right)^{-1}\delta_{q+1}^{1/2}\ell^{-N}\left|k\right|^{-2d}\\
&\hspace{0.2in}+\delta_{q+1}^{1/2}\left\Vert \left(\nabla\overline{v}_{q}\right)\left(\nabla\Phi_{i}\right)^{-1}a_{k}\left(\underline{R_{i}}\circ\mathbf{\Phi}_{i}\right)\right\Vert _{N}\\
&\hspace{0.2in}+\delta_{q+1}^{1/2}\left\Vert \left(\nabla\Phi_{i}\right)^{-1}\left(\nabla a_{k}\left(\underline{R_{i}}\circ\mathbf{\Phi}_{i}\right)D_{t,q}\left(\underline{R_{i}}\circ\mathbf{\Phi}_{i}\right)\right)\right\Vert _{N}.
\end{align*}
This leads to the bound
\begin{align*}
&\left(\epsilon_{q}\tau_{q}\right)^{-1}\delta_{q+1}^{1/2}\ell^{-N}\left|k\right|^{-2d}+\delta_{q+1}^{1/2}\delta_{q}^{\frac{1}{2}}\lambda_{q}\ell^{-N}\left|k\right|^{-2d}+\left(\epsilon_{q}\tau_{q}\right)^{-1}\delta_{q+1}^{1/2}\ell^{-N+\alpha}\left|k\right|^{-2d}\\
&\hspace{0.2in}\lesssim\left(\epsilon_{q}\tau_{q}\right)^{-1}\delta_{q+1}^{1/2}\ell^{-N}\left|k\right|^{-2d}
\end{align*}
which completes the proof of (\ref{eq:per_est7}).

It remains to show (\ref{eq:per_est8}).  For this, we again use (\ref{eq:chain_r}), and write, using schematic notation,
\begin{align*}
& \left\Vert D_{t,q}c_{i,k}\right\Vert _{N}\\
&\hspace{0.2in}\sim\delta_{q+1}^{1/2}\lambda_{q+1}^{-1}\left|k\right|^{-1}\left\Vert \partial_{\underline{t}}\left(\rho_{i}(\underline{t})\left(\nabla X_{i}\right)*\nabla\left(a_{k}\left(\underline{R_{i}}\right)\right)\right)\circ\mathbf{\Phi}_{i}\right\Vert _{N},
\end{align*}
which leads to the bound
\begin{align*}
&\delta_{q+1}^{1/2}\left(\epsilon_{q}\tau_{q}\right)^{-1}\lambda_{q+1}^{-1}\left|k\right|^{-1}\left\Vert \left(\nabla\Phi_{i}\right)^{-1}*\nabla\left(a_{k}\left(\underline{R_{i}}\circ\mathbf{\Phi}_{i}\right)\right)*\nabla\Phi_{i}^{-1}\right\Vert _{N}\\
&\hspace{0.2in} +\delta_{q+1}^{1/2}\lambda_{q+1}^{-1}\left|k\right|^{-1}\left\Vert \nabla\overline{v}_{q}*\nabla\Phi_{i}^{-1}*\nabla\left(a_{k}\left(\underline{R_{i}}\circ\mathbf{\Phi}_{i}\right)\right)*\nabla\Phi_{i}^{-1}\right\Vert _{N}\\
&\hspace{0.2in} +\delta_{q+1}^{1/2}\lambda_{q+1}^{-1}\left|k\right|^{-1}\left\Vert \left(\nabla\Phi_{i}\right)^{-1}*\nabla\left(\nabla a_{k}\left(\underline{R_{i}}\circ\mathbf{\Phi}_{i}\right)D_{t,q}\left(\underline{R_{i}}\circ\mathbf{\Phi}_{i}\right)\right)*\nabla\Phi_{i}^{-1}\right\Vert _{N}.
\end{align*}
This expression is in turn bounded by
\begin{align*}
&\left(\epsilon_{q}\tau_{q}\right)^{-1}\delta_{q+1}^{1/2}\lambda_{q+1}^{-1}\ell^{-N-1}\left|k\right|^{-2d}+\delta_{q+1}^{1/2}\lambda_{q+1}^{-1}\delta_{q}^{\frac{1}{2}}\lambda_{q}\ell^{-N-1}\left|k\right|^{-2d}\\
&\hspace{0.2in}\lesssim \left(\epsilon_{q}\tau_{q}\right)^{-1}\delta_{q+1}^{1/2}\lambda_{q+1}^{-1}\ell^{-N-1}\left|k\right|^{-2d},
\end{align*}
as desired.  This completes the proof of the stated estimates.
\end{proof}

We now record a useful corollary which will imply (\ref{eq:perturb_1}).

\begin{cor}
\label{cor:There-is-} There is $M=M(d)$ (independent of $q$) such
that 
\begin{align}
\|w^{(c)}\|_{0}+\lambda_{q+1}^{-1}\|\nabla w^{(c)}\|_{0} & \lesssim\delta_{q+1}^{1/2}\lambda_{q+1}^{-1}\ell^{-1}\label{eq:cor1}\\
\|w^{(o)}\|_{0}+\lambda_{q+1}^{-1}\|\nabla w^{(o)}\|_{0} & \leq\frac{M}{4}\delta_{q+1}^{1/2}\label{eq:cor2}\\
\|w_{q+1}\|_{0}+\lambda_{q+1}^{-1}\|\nabla w_{q+1}\|_{0} & \leq\frac{M}{2}\delta_{q+1}^{1/2}\label{eq:cor3}
\end{align}
\end{cor}

\begin{proof}
Without loss of generality, we may assume that $a$ is large enough to ensure $\left\Vert \nabla\Phi_{i}\right\Vert _{0}\leq2$. 

Recall that $t\in\left[t_{i}-\epsilon_{q}\tau_{q},t_{i}+2\epsilon_{q}\tau_{q}\right]$
so that $$\|w^{(c)}\|_{1}=\left\Vert \sum_{k\neq0}c_{i,k}e^{i2\pi\left\langle \lambda_{q+1}k,\Phi_{i}\right\rangle }\right\Vert _{1},$$
and (\ref{eq:cor1}) follows immediately from (\ref{eq:per_est4}) and
(\ref{eq:length_Freq}).

From the proof of (\ref{eq:per_est3}), there is $C=C(d)$ (independent
of $q$) such that 
\[
\left\Vert b_{i,k}\right\Vert _{0}\leq C\delta_{q+1}^{\frac{1}{2}}\left|k\right|^{-2d}
\]
Then (\ref{eq:cor2}) and (\ref{eq:cor3}) follow immediately from (\ref{eq:per_est3}),
(\ref{eq:length_Freq}) and (\ref{eq:cor1}).
\end{proof}

\subsection{Stress error estimates}

Suppose $t\in\left[t_{i}-\epsilon_{q}\tau_{q},t_{i}+2\epsilon_{q}\tau_{q}\right]$.  To complete the proof of \Propref{convex_int}
it remains to prove
\begin{equation}
\left\Vert R_{q+1}\right\Vert _{\alpha}\lesssim\epsilon_{q+1}\delta_{q+2}\lambda_{q+1}^{-4\alpha}.\label{eq:fin_stress_er}
\end{equation}

We will often need to use an important antidivergence estimate from
\cite{buckmasterOnsagerConjectureAdmissible2017}, stated in the following lemma.
\begin{lem}[Proposition C.2 in \cite{buckmasterOnsagerConjectureAdmissible2017}]
For any $N\in\mathbb{N}_{1}$, $u\in\mathfrak{X}\left(\mathbb{T}^{d}\right)$, and
$\phi\in C^{\infty}\left(\mathbb{T}^{d}\to\mathbb{T}^{d}\right)$
such that $\frac{1}{2}\leq\left|\nabla\phi\right|\leq2$, we have
\begin{align}
&\left\Vert \mathcal{R}\left(u(x)e^{i2\pi\left\langle k,\phi\right\rangle }\right)\right\Vert _{\alpha}\nonumber\\
&\hspace{0.2in}\lesssim_{N}\left|k\right|^{\alpha-1}\left\Vert u\right\Vert _{0}+\left|k\right|^{\alpha-N}\left(\left\Vert u\right\Vert _{0}\left\Vert \phi\right\Vert _{N+\alpha}+\left\Vert u\right\Vert _{N+\alpha}\right).\label{eq:antidiv_est}
\end{align}
\end{lem}

Another fact we will use often is that when $N$ is chosen large enough (independent of
$q$), we have
\begin{equation}
\ell_{q}^{N+10\alpha}\lambda_{q+1}^{N-1-10\alpha}>1\label{eq:trick2}
\end{equation}
This comes from 
\[
-\beta b+\beta-1-\frac{\sigma}{2}+b\left(\frac{N-1-10\alpha}{N+10\alpha}\right)>0
\]
which is implied by (\ref{eq:intermed_para}) when $N=N\left(b,\beta,\sigma,\alpha\right)$
is large enough. Unless otherwise noted, we will be using this choice of $N$.

\subsubsection{Nash error}

By using (\ref{eq:antidiv_est}) and \Propref{pertub_est}, we have
\begin{align*}
\left\Vert \mathcal{R}\left(w^{(o)}\cdot\nabla\overline{v}_{q}\right)\right\Vert _{\alpha} & \lesssim\sum_{k\neq0}\left\Vert \mathcal{R}\left(b_{i,k}\cdot\nabla\overline{v}_{q}e^{i2\pi\left\langle \lambda_{q+1}k,\Phi_{i}\right\rangle }\right)\right\Vert _{\alpha}\\
 & \lesssim_{N}\sum_{k\neq0}\left|\lambda_{q+1}k\right|^{\alpha-1}\left|k\right|^{-2d}\delta_{q+1}^{\frac{1}{2}}\delta_{q}^{\frac{1}{2}}\lambda_{q}\\
 & \quad+\left|\lambda_{q+1}k\right|^{\alpha-N}\left|k\right|^{-2d}\left(\delta_{q+1}^{\frac{1}{2}}\delta_{q}^{\frac{1}{2}}\lambda_{q}\ell^{-N-2\alpha}\right)\\
 & \lesssim\delta_{q+1}^{\frac{1}{2}}\delta_{q}^{\frac{1}{2}}\lambda_{q+1}^{\alpha-1}\lambda_{q}\lesssim\epsilon_{q+1}\delta_{q+2}\lambda_{q+1}^{-4\alpha}
\end{align*}
where we used (\ref{eq:trick2}) to pass to the last line, and (\ref{eq:stress_size_ind1})
in the last inequality. 

Similarly,
\begin{align*}
\left\Vert \mathcal{R}\left(w^{(c)}\cdot\nabla\overline{v}_{q}\right)\right\Vert _{\alpha} & \lesssim\sum_{k\neq0}\left\Vert \mathcal{R}\left(c_{i,k}\cdot\nabla\overline{v}_{q}e^{i2\pi\left\langle \lambda_{q+1}k,\Phi_{i}\right\rangle }\right)\right\Vert _{\alpha}\\
 & \lesssim_{N}\sum_{k\neq0}\left|\lambda_{q+1}k\right|^{\alpha-1}\left|k\right|^{-2d}\delta_{q+1}^{\frac{1}{2}}\delta_{q}^{\frac{1}{2}}\lambda_{q}\left(\lambda_{q+1}\ell\right)^{-1}\\
 & \quad+\left|\lambda_{q+1}k\right|^{\alpha-N}\left|k\right|^{-2d}\left(\delta_{q+1}^{\frac{1}{2}}\delta_{q}^{\frac{1}{2}}\lambda_{q}\ell^{-N-2\alpha}\right)\left(\lambda_{q+1}\ell\right)^{-1}\\
 & \lesssim\delta_{q+1}^{\frac{1}{2}}\lambda_{q+1}^{\alpha-1}\delta_{q}^{\frac{1}{2}}\lambda_{q}\lesssim\epsilon_{q+1}\delta_{q+2}\lambda_{q+1}^{-4\alpha}
\end{align*}
The only difference is the term $\left(\lambda_{q+1}\ell\right)^{-1}$
which is less than $1$ by (\ref{eq:length_Freq}). Thus we have 
\[
\left\Vert R_{\mathrm{Nash}}\right\Vert _{\alpha}\lesssim\epsilon_{q+1}\delta_{q+2}\lambda_{q+1}^{-4\alpha}
\]

\subsubsection{Transport error}

The important observation here is that $D_{t,q}\left(e^{i2\pi\left\langle \lambda_{q+1}k,\Phi_{i}\right\rangle }\right)=0$,
which helps avoid an extra factor $\lambda_{q+1}$.

Arguing as above, we have 
\begin{align*}
\left\Vert \mathcal{R}D_{t,q}w^{(o)}\right\Vert _{\alpha} & \lesssim\sum_{k\neq0}\left\Vert \mathcal{R}\left(D_{t,q}b_{i,k}e^{i2\pi\left\langle \lambda_{q+1}k,\Phi_{i}\right\rangle }\right)\right\Vert _{\alpha}\\
 & \lesssim_{N}\sum_{k\neq0}\left|\lambda_{q+1}k\right|^{\alpha-1}\left|k\right|^{-2d}\delta_{q+1}^{\frac{1}{2}}\left(\epsilon_{q}\tau_{q}\right)^{-1}\\
 & \quad+\left|\lambda_{q+1}k\right|^{\alpha-N}\left|k\right|^{-2d}\left(\delta_{q+1}^{\frac{1}{2}}\ell^{-N-\alpha}\right)\left(\epsilon_{q}\tau_{q}\right)^{-1}\\
 & \lesssim\delta_{q+1}^{\frac{1}{2}}\lambda_{q+1}^{\alpha-1}\left(\epsilon_{q}\tau_{q}\right)^{-1}=\epsilon_{q}^{-1}\delta_{q+1}^{\frac{1}{2}}\delta_{q}^{\frac{1}{2}}\lambda_{q+1}^{\alpha-1}\lambda_{q}^{1+3\alpha}\\
 & \lesssim\epsilon_{q+1}\delta_{q+2}\lambda_{q+1}^{-4\alpha}
\end{align*}
and 
\begin{align*}
\left\Vert \mathcal{R}D_{t,q}w^{(c)}\right\Vert _{\alpha} & \lesssim\sum_{k\neq0}\left\Vert \mathcal{R}\left(D_{t,q}c_{i,k}e^{i2\pi\left\langle \lambda_{q+1}k,\Phi_{i}\right\rangle }\right)\right\Vert _{\alpha}\\
 & \lesssim_{N}\sum_{k\neq0}\left|\lambda_{q+1}k\right|^{\alpha-1}\left|k\right|^{-2d}\delta_{q+1}^{\frac{1}{2}}\left(\epsilon_{q}\tau_{q}\right)^{-1}\left(\lambda_{q+1}\ell\right)^{-1}\\
 & \quad+\left|\lambda_{q+1}k\right|^{\alpha-N}\left|k\right|^{-2d}\left(\delta_{q+1}^{\frac{1}{2}}\ell^{-N-\alpha}\right)\left(\epsilon_{q}\tau_{q}\right)^{-1}\left(\lambda_{q+1}\ell\right)^{-1}\\
 & \lesssim\delta_{q+1}^{\frac{1}{2}}\lambda_{q+1}^{\alpha-1}\left(\epsilon_{q}\tau_{q}\right)^{-1}=\epsilon_{q}^{-1}\delta_{q+1}^{\frac{1}{2}}\delta_{q}^{\frac{1}{2}}\lambda_{q+1}^{\alpha-1}\lambda_{q}^{1+3\alpha}\\
 & \lesssim\epsilon_{q+1}\delta_{q+2}\lambda_{q+1}^{-4\alpha}
\end{align*}
Thus we have $\left\Vert R_{\mathrm{trans}}\right\Vert _{\alpha}\lesssim\epsilon_{q+1}\delta_{q+2}\lambda_{q+1}^{-4\alpha}$.

\subsubsection{Oscillation error}

We observe that 
\begin{align*}
R_{\mathrm{osc}} & :=\mathcal{R}\Div\left(\overline{R}_{q}+w_{q+1}\otimes w_{q+1}\right)\\
 & =\underbrace{\mathcal{R}\Div\left(\overline{R}_{q}+w^{(o)}\otimes w^{(o)}\right)}_{:=\mathcal{O}_{1}}\\
&\hspace{0.2in}+\underbrace{\mathcal{R}\Div\left(w^{(c)}\otimes w^{(o)}+w^{(o)}\otimes w^{(c)}+w^{(c)}\otimes w^{(c)}\right)}_{:=\mathcal{O}_{2}}
\end{align*}
Then, using \Corref{There-is-}, and the fact that $\mathcal{R}\Div$ is a Calderón-Zygmund operator, we obtain
\begin{align*}
\left\Vert \mathcal{O}_{2}\right\Vert _{\alpha} & \lesssim\left\Vert w^{(c)}\right\Vert _{\alpha}\left\Vert w^{(o)}\right\Vert _{\alpha}+\left\Vert w^{(c)}\right\Vert _{\alpha}^{2}\lesssim\delta_{q+1}\left(\ell\lambda_{q+1}\right)^{-1}\\
 & =\epsilon_{q}^{-\frac{1}{2}}\delta_{q+1}^{\frac{1}{2}}\delta_{q}^{\frac{1}{2}}\lambda_{q+1}^{-1}\lambda_{q}^{1+\frac{3\alpha}{2}}\lesssim\epsilon_{q+1}\delta_{q+2}\lambda_{q+1}^{-4\alpha}
\end{align*}
where we have once again used (\ref{eq:stress_size_ind1}). 

On the other hand, because $\rho_{i}^{2}=1$ on $\supp\overline{R}_{q}$,
we have 
\begin{align*}
\mathcal{O}_{1} & =\mathcal{R}\Div\left(\overline{R}_{q}+\delta_{q+1}\rho_{i}^{2}\mathbf{\Phi}_{i}^{*}\left(W(\underline{R_{i}},\lambda_{q+1}\cdot)\otimes W(\underline{R_{i}},\lambda_{q+1}\cdot)\right)\right)\\
 & =\mathcal{R}\Div\bigg(\overline{R}_{q}+\delta_{q+1}\rho_{i}^{2}\bigg(\textrm{Id}-\frac{\overline{R}_{q}}{\delta_{q+1}}\bigg)\\
&\hspace{1.2in}+\delta_{q+1}\rho_{i}^{2}\mathbf{\Phi}_{i}^{*}\bigg(\sum_{k\in\mathbb{Z}^{d}\backslash\{0\}}C_{k}\left(\underline{R_{i}}\right)e^{i2\pi\left\langle \lambda_{q+1}k,\cdot\right\rangle }\bigg)\bigg)\\
 & =\sum_{k\in\mathbb{Z}^{d}\backslash\{0\}}\delta_{q+1}\rho_{i}^{2}\mathcal{R}\Div\left(\mathbf{\Phi}_{i}^{*}\left(C_{k}\left(\underline{R_{i}}\right)\right)e^{i2\pi\left\langle \lambda_{q+1}k,\Phi_{i}\right\rangle }\right)\\
 & =\sum_{k\in\mathbb{Z}^{d}\backslash\{0\}}\delta_{q+1}\rho_{i}^{2}\mathcal{R}\left(\Div\left(\mathbf{\Phi}_{i}^{*}\left(C_{k}\left(\underline{R_{i}}\right)\right)\right)e^{i2\pi\left\langle \lambda_{q+1}k,\Phi_{i}\right\rangle }\right)\\
&\hspace{0.8in}+\delta_{q+1}\rho_{i}^{2}\mathcal{R}\underbrace{\left(d_{x}\left(e^{i2\pi\left\langle \lambda_{q+1}k,\Phi_{i}\right\rangle }\right)\lrcorner\mathbf{\Phi}_{i}^{*}\left(C_{k}\left(\underline{R_{i}}\right)\right)\right)}_{\mathcal{O}_{3}}
\end{align*}
We note that
\begin{align*}
\mathcal{O}_{3} & =d_{x}\left(\mathbf{\Phi}_{i}^{*}e^{i2\pi\left\langle \lambda_{q+1}k,\cdot\right\rangle }\right)\lrcorner\mathbf{\Phi}_{i}^{*}\left(C_{k}\left(\underline{R_{i}}\right)\right)=\mathbf{\Phi}_{i}^{*}\left(\left(d_{\underline{x}}e^{i2\pi\left\langle \lambda_{q+1}k,\cdot\right\rangle }\right)\lrcorner C_{k}\left(\underline{R_{i}}\right)\right)=0
\end{align*}
because of (\ref{eq:kflat_c}). Then because of (\ref{eq:antidiv_est})
and (\ref{eq:stress_size_ind1}):
\begin{align*}
\left\Vert \mathcal{O}_{1}\right\Vert _{\alpha} & \lesssim\sum_{k\in\mathbb{Z}^{d}\backslash\{0\}}\left\Vert \delta_{q+1}\mathcal{R}\left(\Div\left(\nabla\Phi_{i}^{-1}C_{k}\left(\underline{R_{i}}\circ\mathbf{\Phi}_{i}\right)\nabla\Phi_{i}^{-T}\right)e^{i2\pi\left\langle \lambda_{q+1}k,\Phi_{i}\right\rangle }\right)\right\Vert _{\alpha}\\
 & \lesssim_{N}\sum_{k\neq0}\left|\lambda_{q+1}k\right|^{\alpha-1}\left|k\right|^{-2d}\delta_{q+1}\ell^{-1}+\left|\lambda_{q+1}k\right|^{\alpha-N}\left|k\right|^{-2d}\left(\delta_{q+1}\ell^{-N-3\alpha}\right)\ell^{-1}\\
 & \lesssim\lambda_{q+1}^{\alpha-1}\delta_{q+1}\ell^{-1}=\epsilon_{q}^{-\frac{1}{2}}\delta_{q+1}^{\frac{1}{2}}\delta_{q}^{\frac{1}{2}}\lambda_{q+1}^{\alpha-1}\lambda_{q}^{1+\frac{3\alpha}{2}}\lesssim\epsilon_{q+1}\delta_{q+2}\lambda_{q+1}^{-4\alpha}
\end{align*}
Therefore $\left\Vert R_{\mathrm{osc}}\right\Vert _{\alpha}\lesssim\epsilon_{q+1}\delta_{q+2}\lambda_{q+1}^{-4\alpha}$.

\subsubsection{Dissipative error}

Without loss of generality, we may assume $2\alpha+2\gamma<1$ (by choosing $\alpha$ sufficiently small). Because $\mathcal{R}$
and $\left(-\Delta\right)^{\gamma}$ commute, and because $\left(-\Delta\right)^{\gamma}$
is a bounded map from $C^{2\gamma+2\alpha}$ to $C^{\alpha}$ (\cite[Theorem B.1]{derosaInfinitelyManyLeray2019}),
we have
\begin{align*}
\left\Vert R_{\mathrm{dis}}\right\Vert _{\alpha} & \lesssim\left\Vert \mathcal{R}w_{q+1}\right\Vert _{2\alpha+2\gamma}\lesssim\left\Vert \mathcal{R}w_{q+1}\right\Vert _{1}^{2\gamma+2\alpha}\left\Vert \mathcal{R}w_{q+1}\right\Vert _{0}^{1-2\gamma-2\alpha}.
\end{align*}
Then because $\nabla\mathcal{R}$ is a Calderón-Zygmund operator,
and because of \Corref{There-is-}.:
\[
\left\Vert \mathcal{R}w_{q+1}\right\Vert _{1}\lesssim\left\Vert \nabla\mathcal{R}w_{q+1}\right\Vert _{0}\leq\left\Vert \nabla\mathcal{R}w_{q+1}\right\Vert _{\alpha}\lesssim\left\Vert w_{q+1}\right\Vert _{\alpha}\lesssim\delta_{q+1}^{\frac{1}{2}}\lambda_{q+1}^{\alpha}
\]
Meanwhile, because of (\ref{eq:antidiv_est}) : 
\begin{align*}
\left\Vert \mathcal{R}w_{q+1}\right\Vert _{\alpha} & =\sum_{k\neq0}\left\Vert \mathcal{R}\left(\left(b_{i,k}+c_{i,k}\right)e^{i2\pi\left\langle \lambda_{q+1}k,\Phi_{i}\right\rangle }\right)\right\Vert _{\alpha}\\
 & \lesssim_{N}\sum_{k\neq0}\left|\lambda_{q+1}k\right|^{\alpha-1}\left|k\right|^{-2d}\delta_{q+1}^{\frac{1}{2}}+\left|\lambda_{q+1}k\right|^{\alpha-N}\left|k\right|^{-2d}\delta_{q+1}^{\frac{1}{2}}\ell^{-N-\alpha}\\
 & \lesssim\delta_{q+1}^{\frac{1}{2}}\lambda_{q+1}^{\alpha-1}
\end{align*}
Therefore:
\[
\left\Vert R_{\mathrm{dis}}\right\Vert _{\alpha}\lesssim\delta_{q+1}^{\frac{1}{2}}\lambda_{q+1}^{\alpha\left(2\gamma+2\alpha\right)+\left(\alpha-1\right)\left(1-2\gamma-2\alpha\right)}\lesssim\epsilon_{q+1}\delta_{q+2}\lambda_{q+1}^{-4\alpha}
\]
because of (\ref{eq:stress_size_ind3}), when $\alpha=\alpha\left(\sigma,b,\beta,\gamma\right)$
is small enough. This completes the proof of (\ref{eq:fin_stress_er}), and therefore of Proposition \ref{prop:convex_int}.

\appendix

\section{\label{app:Geometric-preliminaries}Geometric preliminaries}

We recall the Hodge decomposition
\[
\mathrm{Id}=\mathcal{P}_{1}+\mathcal{P}_{2}+\mathcal{P}_{3}
\]
where $\mathcal{P}_{1}:=d\left(-\Delta\right)^{-1}\delta$ and $\mathcal{P}_{2}:=\delta\left(-\Delta\right)^{-1}d$
and $\mathcal{P}_{3}$ maps to harmonic forms (cf. \cite[Section 5.8]{Taylor_PDE1}).
We observe that $\mathcal{P}_{1},\mathcal{P}_{2}$ are Calderón-Zygmund
operators. We also recall that $\delta=-\Div$, where $\left(\Div T\right)^{i_{1}...i_{k}}=\nabla_{j}T^{ji_{1}...i_{k}}$
for any tensor $T$.

Due to the musical isomorphism, the Hodge projections $\mathcal{P}_{i}$
are also defined on vector fields, and we also write $\sharp\mathcal{P}_{i}\flat$
as $\mathcal{P}_{i}$ for convenience (unless ambiguity arises).

Because the torus is flat, we have the identities 
\begin{align}
\delta\flat\left(X\cdot\nabla Y\right) & =\delta\flat\left(Y\cdot\nabla X\right)\nonumber \\
\mathcal{P}_{1}\left(X\cdot\nabla Y\right) & =\mathcal{P}_{1}\left(Y\cdot\nabla X\right)\label{eq:ident_P1}
\end{align}
for any divergence-free vector fields $X,Y$. On the torus, harmonic
1-forms (or vector fields) are precisely those which have mean zero.

\begin{defn}[Time-dependent Lie derivative]
 For any smooth family of diffeomorphisms $\left(F_{t}\right)$ and
differential forms $\left(\alpha_{t}\right)$ we have 
\[
\partial_{t}\left(F_{t}^{*}\alpha_{t}\right)=F_{t}^{*}\left(\mathcal{L}_{X_{t}}\alpha_{t}+\partial_{t}\alpha_{t}\right)
\]
where $\left(X_{t}\right)$ is a time-dependent vector field defined
by $\partial_{t}F_{t}=X_{t}\circ F_{t}$.
\end{defn}

\begin{lem}
For any diffeomorphism $\Phi$, vector field $u$ and differential
form $\alpha$, we recall the pullback identity:
\begin{align}
\Phi^{*}\left(\mathcal{L}_{u}\alpha\right) & =\mathcal{L}_{\Phi^{*}u}\Phi^{*}\alpha\label{eq:pullback_lie}
\end{align}
\end{lem}

\begin{rem}
The pullback of a 1-form has a different meaning from the pullback
of a vector field, and we do not have $\Phi^{*}\flat X=\flat\Phi^{*}X$
unless $\Phi$ is an isometry. 
\end{rem}

We conclude this appendix by introducing several operators that play a key role in our analysis.  In particular,
we will make use of the antidivergence operator 
$$\mathcal{R}:C^{\infty}\left(\mathbb{T}^{d},\mathbb{R}^{d}\right)\to C^{\infty}\left(\mathbb{T}^{d},\mathcal{S}_{0}^{d\times d}\right),$$ given by
\begin{align}
\left(\mathcal{R}v\right)_{ij} & =\mathcal{R}_{ijk}v^{k},\label{eq:antidiv}
\end{align}
with
\begin{align}
\mathcal{R}_{ijk} & :=-\frac{d-2}{d-1}\Delta^{-2}\partial_{i}\partial_{j}\partial_{k}-\frac{1}{d-1}\Delta^{-1}\partial_{k}\delta_{ij}+\Delta^{-1}\partial_{i}\delta_{jk}+\Delta^{-1}\partial_{j}\delta_{ik}.\nonumber 
\end{align}
Note that $\div\mathcal{R}v=v-\dashint_{\mathbb{T}^{d}}v=\left(1-\mathcal{P}_{3}\right)v$
for any vector field $v$. Moreover, using the musical isomorphism, the operator $\mathcal{R}$
can also be defined on 1-forms, and we will often write $\mathcal{R}\sharp$ as $\mathcal{R}$ 
to simplify notation.

We also define the higher-dimensional analogue of the Biot-Savart operator as $\mathcal{B}:=\left(-\Delta\right)^{-1}d\flat$,
mapping from vector fields to 2-forms. We then have 
\[
\sharp\delta\mathcal{B}=\mathcal{P}_{2}
\]
which implies $\sharp\delta\mathcal{B}v=v-\dashint_{\mathbb{T}^{d}}v=\mathcal{P}_{2}v$
for any divergence-free vector field $v$.

\end{document}